\documentclass[10pt,reqno]{amsart}
\usepackage{graphicx}
\usepackage[utf8]{inputenc}
\usepackage[T1]{fontenc} 
\usepackage{indentfirst,csquotes}

\usepackage{amssymb,amsthm,amsmath}
\usepackage{xcolor,paralist,hyperref,titlesec,fancyhdr,etoolbox}
\newtheorem{theorem}{Theorem}
\newtheorem{definition}{Definition}

\newtheorem{lemma}{Lemma}
\newtheorem{remark}{Remark}
\newtheorem{assumption}{Assumption}
\newtheorem{proposition}{Proposition}
\newtheorem{corollary}{Corollary}

\titleformat{\section}[display]{\normalfont\huge\bfseries\centering}{\centering\chaptertitlename\thechapter}{10pt}{\Large}
\titlespacing*{\section}{0pt}{0ex}{0ex}

\hypersetup{ colorlinks=true, linkcolor=black, filecolor=black, urlcolor=black }

\usepackage{booktabs}
\usepackage{lipsum}

\begin{document}
\title[Recovering the reaction coefficient in a time-fractional diffusion equation..]{Recovering the reaction coefficient in a time-fractional diffusion equation with variable coefficients in any  spatial dimension} 
\author[D.K.Durdiev]{D.K.Durdiev} 
\date{\today}
\address{$ ^1$Bukhara Branch of Romanovskii Institute of Mathematics,
Uzbekistan Academy of Sciences, Bukhara, Uzbekistan}
\address{$^2$Bukhara State University, Bukhara, 705018 Uzbekistan}
\email{durdiyevdd@gmail.com} 
\maketitle

\let\thefootnote\relax

{\bf Abstract}
We explore  the inverse problem of determining the reaction coefficient $q(x',t)$ in the time-fractional diffusion equation ${}^CD_t u-Lu+q(x',t)u=f$ in $\mathbb{R}^n\times(0,T]$, $x=(x',x_n)$, $0<\alpha<1$, from the trace of the solution on the hyperplane $\{x_n=0\}$. Here ${}^CD_t$ is the Gerasimov--Caputo derivative and $L$ is a uniformly elliptic second-order operator with variable, $x_n$-independent coefficients and without zero-order term. For $L=\Delta$ this problem was studied in dimensions $n=1,2,3$ by methods based on pointwise bounds of the kernels, which are not available for $n\ge4$. Our approach relies on the Eidelman--Kochubei fundamental solution: although its singularity on the diagonal depends on the dimension, we show that it satisfies integral bounds of the same form for all $n\ge1$. Using these bounds, we prove existence, uniqueness and stability of regular solutions of the direct problem with a H\"{o}lder continuous reaction coefficient. Without assuming additional smoothness of the solution, we show that the inverse problem is equivalent to a closed system of integral equations for $q$ and the derivatives $\partial_{x_n}u$, $\partial^2_{x_n}u$. For the inverse problem we establish a conditional Lipschitz stability estimate on the whole time interval and global uniqueness in every dimension, and prove local existence of a solution. In the case $a_{in}\equiv0$, $b_n\equiv0$ we obtain a stability estimate with explicit constants, which reduces to the known estimates for the Laplacian. An explicit two-dimensional example with a mixed coefficient and a normal drift shows that the hypotheses of the main theorems are compatible and that the corresponding terms in the trace condition cannot be omitted.
\bigskip


\section*{1. Introduction}

Equations with a fractional time derivative of order $\alpha\in(0,1)$ are a standard model of subdiffusion, that is, of transport processes in which the mean square displacement of a particle grows like $t^{\alpha}$ rather than $t$. Such processes are observed in fractal and strongly heterogeneous media, porous materials and amorphous semiconductors; see the survey \cite{MK00}, the paper
\cite{SW89} and the monographs \cite{Pod99,KST}. In this paper we study the equation
\begin{equation*}
{}^CD_t u-Lu+q(x',t)u=f\qquad\text{in }\mathbb{R}^n\times(0,T],
\end{equation*}
where ${}^CD_t$ is the Gerasimov--Caputo derivative of order $\alpha\in(0,1)$, $L$ is a uniformly elliptic second-order operator with variable coefficients and without zero-order term, and $q$ is a reaction (absorption) coefficient depending on time and on the tangential variables
$x'=(x_1,\dots,x_{n-1})$. We are concerned with the Cauchy problem for this equation (the direct problem) and with the inverse problem of recovering $q$ from the trace $u(x',0,t)=g(x',t)$ of the solution on the hyperplane $\{x_n=0\}$.

\smallskip\noindent\emph{The direct problem.}
For $L=\Delta$ the fundamental solution of the Cauchy problem was found by Schneider and Wyss \cite{SW89} and by Kochubei \cite{Ko90} in
terms of Fox's $H$-functions; a different representation, valid also for diffusion-wave equations, was obtained by Pskhu \cite{Psk09}.
For operators with variable coefficients, Eidelman and Kochubei \cite{EK} constructed the Green matrix $(Z,Y)$ of the Cauchy problem by
the Levi parametrix method and obtained pointwise estimates of $Z$, $Y$ and of their derivatives; see also the monograph \cite{EIK04} and
the extension to diffusion-wave equations \cite{Koc14}. A feature that has no analogue in the classical parabolic theory
\cite{Fri64,LSU} is that for $n\ge2$ the fundamental solution is singular not only at $t=0$ but also on the diagonal $x=\xi$ for every
$t>0$, and the order of this singularity depends on the dimension. Other approaches to time-fractional diffusion equations include
eigenfunction expansions in bounded domains \cite{SY11,KRY20,Jin21}, the maximum principle \cite{Luc09} and decay estimates for
equations in the whole space \cite{KSVZ16}.

\smallskip\noindent\emph{Inverse problems.}
Coefficient inverse problems for parabolic equations with unknown coefficients depending on time, and with additional information given on
a part of the boundary or on an interior manifold, have a long history; we refer to the monographs \cite{POV00,Iva03}. For time-fractional
equations the theory of inverse problems has developed rapidly; see the survey \cite{JR15} and the books \cite{Jin21,KR23,DD25}. Uniqueness for
the simultaneous determination of the order of the derivative and of a spatially varying coefficient was proved in \cite{CNYY09}, and the
recovery of a space-dependent potential was studied in \cite{ZZ17,KR19,JZ21a,JZ21b}. The inverse problem considered in this paper was
studied for $L=\Delta$ by D.\,K.~Durdiev \cite{DD21} ($n=1$), by D.\,K.~Durdiev, A.\,A.~Rahmonov and Z.\,R.~Bozorov \cite{DRB21} ($n=2$)
and by U.\,D.~Durdiev \cite{UD21} ($n=3$). In these works the inverse problem is reduced to a system of integral equations, and local
existence, global uniqueness and a stability estimate are proved. The arguments use the explicit form of the kernels for $L=\Delta$ and
pointwise bounds that are specific to each dimension. Since the kernel $Y$ is unbounded near the diagonal for $n\ge4$
(Table~\ref{tab:sing}), such arguments do not extend to higher dimensions, nor to operators with variable coefficients.  A systematic exposition of coefficient determination problems for time-fractional diffusion
equations, including problems for the Cauchy problem with overdetermination conditions, is given
in the monograph \cite{DD25}.

\smallskip\noindent\emph{Main results.}
The aim of the paper is to treat the direct and the inverse problem for a general operator $L$ in an arbitrary dimension $n\ge1$.
Our main results are as follows.
 \begin{enumerate}
\item We collect the estimates of the Eidelman--Kochubei kernels (Section 3), with attention to the fact that their pointwise
behaviour differs for $n=1,2,3,4$ and $n\ge5$, and show that after normalisation they are dominated by a single profile; this gives integral
bounds of the same form in all dimensions (Lemmas \ref{lem:unified} and \ref{lem:int}).
   \item For the direct problem with a H\"{o}lder continuous reaction coefficient we prove existence and uniqueness of a regular solution,
a priori estimates and a stability estimate (Theorems~\ref{thm:direct} and \ref{thm:stab}). The solution depends on the reaction
coefficient only through its sup-norm.
      \item We prove that the solution is twice differentiable in the normal variable $x_n$ (Theorem~\ref{thm:diff}) instead of assuming it,
show that the inverse problem is equivalent to a problem for $\partial_{x_n}u$, $\partial^2_{x_n}u$ with a trace condition on
$\{x_n=0\}$ (Theorem~\ref{thm:equiv}), and reduce it to a closed system of integral equations (Proposition~\ref{prop:system}).
         \item For the inverse problem we prove a conditional Lipschitz stability estimate in the sup-norm on the whole interval $[0,T]$
(Theorem~\ref{thm:A}) and global uniqueness (Theorem~\ref{thm:B}). These results hold in every dimension
and also for operators with mixed coefficients $a_{in}$ and a normal drift $b_n$; in that case the trace condition contains tangential
derivatives of $\partial_{x_n}u$, and the resulting Volterra operator has order $\alpha/2$ instead of $\alpha$.
            \item We prove the local existence in the general case (Theorem~\ref{thm:Dgen}). When $a_{in}\equiv0$ and $b_n\equiv0$ we also obtain a stability estimate with explicit constants, which reduces to the estimates of
\cite{DD21,DRB21,UD21} for $L=\Delta$ (Theorem~\ref{thm:C}), and local existence of a solution (Theorem~\ref{thm:D}). 
 \end{enumerate} 

\smallskip\noindent\emph{Organisation of the paper.}
In Section 2 we state the problem and the assumptions and introduce the class of regular solutions.
Section 3 recalls the results of \cite{EK} and derives the dimension-independent bounds. The direct problem is studied in
Section 4, and the reduction of the inverse problem is carried out in Section 5. Section 6
contains the formulation of the main results on the inverse problem, Section 7 auxiliary lemmas, and
Sections8--10 the proofs. Section 11 presents an example illustrating the applicability of the obtained results. Section 12 contains a discussion of the results and open questions.

\section*{2. Statement of the problem and standing assumptions}\label{sec:statement}

Let $n\ge 1$, $T>0$, $0<\alpha<1$. We denote a point $x\in  \mathbb{R}^n$ by splitting it into its first $n-1$ coordinates and its last coordinate, writing $x=(x',x_n)$ with
$x'=(x_1,\dots,x_{n-1})\in\mathbb{R}^{n-1}$; for $n=1$ the variable $x'$ is absent.
We write $e_n=(0,\dots,0,1)\in \mathbb{R}^n$ for the last standard basis vector, so that $x+he_n=(x',x_n+h)$, $h\in \mathbb{R}$.
Let  $\Pi_T= \mathbb{R}^n\times(0,T]$ and  $\overline\Pi_T= \mathbb{R}^n\times[0,T]$ denote the spacetime domains.

We adopt the definition of the regularized (Gerasimov--Caputo) fractional derivative from \cite{EK}:
\begin{equation}\label{eq:caputo}
({}^CD_t u)(x,t)=\frac{\partial}{\partial t}\,
I^{1-\alpha}_{0+}\bigl[u(x,\cdot)-u(x,0)\bigr](t),\qquad
(I^{\beta}_{0+}\varphi)(t)=\frac{1}{\Gamma(\beta)}\int_0^t\frac{\omega(\tau)d\tau}{(t-\tau)^{1-\beta}},
\end{equation}
which coincides with $\Gamma(1-\alpha)^{-1}\int_0^t(t-\tau)^{-\alpha}u_\tau(x,\tau)d\tau$
for functions absolutely continuous in $t,$ where $\Gamma(\cdot)$ is the Euler's gamma function.

Consider the Cauchy problem
\begin{align}
&({}^CD_t u)(x,t)-Lu(x,t)+q(x',t)\,u(x,t)=f(x,t),&&(x,t)\in\Pi_T,\label{eq:main}\\
&u(x,0)=\varphi(x),&& x\in \mathbb{R}^n,\label{eq:init}
\end{align}
where
\begin{equation*}
L=\sum_{i,j=1}^n a_{ij}(x)\frac{\partial^2}{\partial x_i\partial x_j}+\sum_{j=1}^n b_j(x)\frac{\partial}{\partial x_j}.
\end{equation*}

\smallskip\noindent
For this problem we explore the inverse problem of finding  $q(x',t)$ such that the solution of \eqref{eq:main}, \eqref{eq:init} satisfies
\begin{equation}\label{eq:over}
u(x',0,t)=g(x',t),\qquad x'\in \mathbb{R}^{n-1},\ t\in(0,T].
\end{equation}
For $n=1$ this reads $q=q(t)$, $u(0,t)=g(t)$.

For $L=\Delta$ this is precisely the problem of
\cite{DD21} ($n=1$), \cite{DRB21} ($n=2$) and \cite{UD21} ($n=3$).

\subsection*{2.1 Assumptions on the operator $L$}

\begin{assumption}\label{ass:A} 

(A1) \quad $a_{ij}=a_{ji}$ and $b_j$ are bounded real functions, H\"{o}lder continuous on $\mathbb{R}^n$
with some exponent $\gamma\in(0,1)$, and there exists  $\delta>0$ such that
$\sum_{i,j}a_{ij}(x)\zeta_i\zeta_j\ge\delta|\zeta|^2$ holds  for all $x,\zeta\in\mathbb{R}^n$.\\
(A2) \quad The coefficients do not depend on $x_n$: $a_{ij}=a_{ij}(x')$, $b_j=b_j(x')$. 

We shall also single out the special case\\
(A3) \quad $a_{in}\equiv0$ for $1\le i\le n-1$ and $b_n\equiv0$, i.e.\ $L=L'+a_{nn}(x')\partial_{x_n}^2$, 
where, throughout this paper
\begin{equation}\label{eq:Lprime}
\quad\quad\quad\quad\quad\quad\quad\quad\quad\quad L'=\sum_{i,j=1}^{n-1}a_{ij}(x')\frac{\partial^2}{\partial x_i\partial x_j}+\sum_{j=1}^{n-1}b_j(x')\frac{\partial}{\partial x_j}
\end{equation}
is the ``tangential part'' of $L$. Under (A1), $L'$ is uniformly elliptic in $\mathbb{R}^{n-1}$ (the
matrix $(a_{ij})_{i,j<n}$ is a principal submatrix of a uniformly positive matrix), so $L'$ satisfies (A1) in
$\mathbb{R}^{n-1}$. In particular, the Laplacian $L=\Delta$ satisfies (A1)--(A3).
\end{assumption}

\begin{remark}\label{rem:A2}
The role of hypothesis (A2) in Section 5 is manifested through the invariance of the problem under translations in the $x_{n}$-coordinate, as established in Lemma~\ref{lem:transl}.\end{remark}

\subsection*{2.2. Function spaces}

For $n\ge1$ let $C_b(\mathbb{R}^n)$ be the bounded continuous functions with $|\psi|_0=\sup_{\mathbb{R}^n}|\psi|$, and for
$\ell\in(0,1)$
\begin{equation*}
[\psi]_{\ell}=\sup_{x\ne\tilde x}\frac{|\psi(x)-\psi(\tilde x)|}{|x-\tilde x|^{\ell}},\qquad
H^{\ell}(\mathbb{R}^n)=\{\psi\in C_b(\mathbb{R}^n):\ |\psi|_\ell:=|\psi|_0+[\psi]_{\ell}<\infty\}.    
\end{equation*} 
$H^{k+\ell}(\mathbb{R}^n)$ ($k\in\mathbb N$) consists of functions whose derivatives of order $\le k$ belong to
$H^{\ell}(\mathbb{R}^n)$, with the usual norm. In  $\mathbb{R}^n\times[0,T]$ we use
\begin{equation*}
C^{\ell,0}_b(\mathbb{R}^n\times[0,T])=\Bigl\{h(x, t)\in C_b(\mathbb{R}^n\times[0,T]):\ \|h\|_\ell:=\|h\|_0+\sup_{t\in[0,T]} [h(\cdot,t)]_{\ell}<\infty\Bigr\},    
\end{equation*} 
where $\|h\|_0=\sup_{\mathbb{R}^n\times[0,T]}|h|$,
and the time-dependent sup norm $\|h\|(t)=\sup_{x\in\mathbb{R}^n}|h(x,t)|$
(we use global H\"{o}lder seminorms, which are equivalent to the norms of the spaces in \cite{LSU} with a localization radius \(\rho _{0}\) considered in \cite{DD21,UD21}).

Following \cite{EK}, $H^{\alpha+\lambda}_{\nu}[0,T]$ denotes the class of functions $\theta$ on $[0,T]$ such that
$t^{\nu}\theta(t)$ is H\"{o}lder  continuous on $[0,T]$ with exponent $\alpha+\lambda$. Throughout the paper we shall assume
\begin{equation}\label{eq:lamnu}
\qquad\qquad\qquad\qquad\qquad\qquad 0<\lambda<1-\alpha,\qquad \alpha+\lambda<\nu<1+\lambda.
\end{equation}

\subsection*{2.3. Formulation of the generic direct problem and regular solutions}

Both the problem for $u$ and the problems obtained after differentiation in $x_n$, and also a problem in
$\mathbb{R}^{n-1}$ appearing in Section 5, are of the form
\begin{equation}\label{eq:P}
\qquad\qquad\qquad\qquad {}^CD_t z-\mathcal L z+p(x,t)\,z=F(x,t)\ \text{ in } \mathbb{R}^N\times(0,T],\qquad z(x,0)=\psi(x),
\tag{P}
\end{equation}
where $\mathcal L$ is an operator of \cite{EK} (with $c(x)=0$) in $\mathbb{R}^N$, and the coefficient $p$
may depend on all variables. Hence, we analyze \eqref{eq:P} in full generality.

\begin{definition}\label{def:reg}
$\mathcal{U}_T(\mathbb{R}^N)$ is the set of $z\in C_b(\mathbb{R}^N\times[0,T])$ such that\\
\qquad\qquad\qquad (a) $\sup_{t\in[0,T]}[z(\cdot,t)]_{\ell'}<\infty$ for some $\ell'\in(0,1)$ (this requirement is dropped if $N=1$);\\
\qquad\qquad\qquad (b) $z(\cdot,t)\in C^2(\mathbb{R}^N)$ for each $t\in(0,T]$;\\
\qquad\qquad\qquad (c) for each $x$, $I^{1-\alpha}_{0+}[z(x,\cdot)-z(x,0)]\in C^1(0,T]$;\\
\qquad\qquad\qquad (d) for each $x$, $z(x,\cdot)\in H^{\alpha+\lambda}_\nu[0,T]$ with $\lambda,\nu$ fixed in \eqref{eq:lamnu}.
\end{definition}

\section*{3. The fundamental solution}\label{sec:fs}

\subsection*{3.1. The Green matrix}
Let $\mathcal L$ satisfy (A1) in $\mathbb{R}^N$ for some $N\in \mathbb{N}$. We recall the results of \cite{EK} that are used below.

\medskip
\noindent\textbf{The parametrix $(Z_0,Y_0)$.}
Fix $\xi\in \mathbb{R}^N$ and let $\mathcal L_\xi$ denote the operator obtained from
$\mathcal L$ by freezing the leading coefficients at the point $\xi$ and dropping
the lower-order terms:
\begin{equation*}
 \qquad\qquad\qquad\qquad\qquad\qquad\qquad \mathcal L_\xi u=\mathbb D_t^{(\alpha)}u-\sum_{i,j=1}^N a_{ij}(\xi)\,
  \frac{\partial^2 u}{\partial x_i\partial x_j}.
\end{equation*}
Put $A(\xi)=(a_{ij}(\xi))_{i,j=1}^N$, $A^{-1}(\xi)=(A^{(ij)}(\xi))_{i,j=1}^N$ and
\begin{equation*}
 \qquad\qquad\qquad\qquad\qquad\qquad \rho_\xi(y)=\sum_{i,j=1}^N A^{(ij)}(\xi)\,y_iy_j=\langle A^{-1}(\xi)y,y\rangle,
  \qquad y\in \mathbb{R}^N .
\end{equation*}
By the uniform ellipticity and boundedness of $a_{ij}$ there exist constants
$0<c_1\le c_2$, independent of $\xi$, such that
$c_1|y|^2\le\rho_\xi(y)\le c_2|y|^2$.

We use Fox's $H$-function of the type $H^{20}_{12}$ \cite[Ch 1]{Math}:
\begin{equation}\label{eq:H2012}
 \qquad\qquad\qquad\qquad H^{20}_{12}\!\left[z\,\middle|\,
  \begin{matrix}(c,\gamma)\\ (\tfrac N2,1),\,(1,1)\end{matrix}\right]
  =\frac1{2\pi i}\int_{\mathcal C}
  \frac{\Gamma(\tfrac N2+s)\,\Gamma(1+s)}{\Gamma(c+\gamma s)}\,z^{-s}\,ds ,
  \qquad z>0,
\end{equation}
where $\mathcal C$ is a loop beginning at $-\infty-i\sigma$ and ending at
$-\infty+i\sigma$ ($\sigma>0$) that encircles all poles of
$\Gamma(\tfrac N2+s)\Gamma(1+s)$. We need to consider only the cases
$(c,\gamma)=(1,\alpha)$ and $(c,\gamma)=(\alpha,\alpha)$.

The Green matrix $(Z_0,Y_0)$ of the Cauchy problem for $\mathcal L_\xi$ is given
explicitly by
\begin{align}
\qquad\qquad\qquad  Z_0(t,y;\xi)&=\frac{\pi^{-N/2}}{(\det A(\xi))^{1/2}}\,\rho_\xi(y)^{-N/2}\,
  H^{20}_{12}\!\left[\frac{\rho_\xi(y)}{4t^{\alpha}}\,\middle|\,
  \begin{matrix}(1,\alpha)\\ (\tfrac N2,1),\,(1,1)\end{matrix}\right],
  \label{eq:Z0}\\[1mm]
\qquad\qquad\qquad  Y_0(t,y;\xi)&=\frac{\pi^{-N/2}}{(\det A(\xi))^{1/2}}\,\rho_\xi(y)^{-N/2}\,
  t^{\alpha-1}\,
  H^{20}_{12}\!\left[\frac{\rho_\xi(y)}{4t^{\alpha}}\,\middle|\,
  \begin{matrix}(\alpha,\alpha)\\ (\tfrac N2,1),\,(1,1)\end{matrix}\right],
  \label{eq:Y0}
\end{align}
$t>0$, $y\in \mathbb{R}^N\setminus\{0\}$. In the theorem below they are taken with
$y=x-\xi$, i.e. $Z_0(t,x-\xi;\xi)$ and $Y_0(t,x-\xi;\xi)$.

Note that $Y_0=D^{1-\alpha}_{0+}Z_0$, where $D^{1-\alpha}_{0+}$ is the Riemann--Liouville derivative in $t$.
For $\alpha=1$ one has 
$$
H^{20}_{12}\!\left[z\,\middle|\,
  \begin{matrix}(1,1)\\ (\tfrac N2,1),\,(1,1)\end{matrix}\right]=z^{N/2}e^{-z},
$$
and 
\eqref{eq:Z0} reduces to the Gaussian kernel 
$(4\pi t)^{-N/2}(\det A(\xi))^{-1/2}\exp\{-\rho_\xi(y)/(4t)\}$.

\begin{theorem}[\cite{EK}]\label{thm:EK}
There exist functions $Z(t,x,\xi)$, $Y(t,x,\xi)$, $t\in(0,T]$, $x\ne\xi$, with the following properties:\\
(i) $Z=Z_0(t,x-\xi;\xi)+V_Z(t,x,\xi)$, $Y=Y_0(t,x-\xi;\xi)+V_Y(t,x,\xi)$, where $Z_0$, $Y_0$ are given by \eqref{eq:Z0}, \eqref{eq:Y0}, and $V_Z$, $V_Y$ satisfy the estimates of \cite[Theorem (a)(i)--(vii)]{EK}.\\
(ii) If $\psi$ is bounded and continuous (locally H\"{o}lder if $N>1$) and $G$ is bounded, continuous on $\mathbb{R}^N\times[0,T]$ and
locally H\"{o}lder continuous in $x$, then
\begin{equation}\label{eq:repr}
\begin{gathered}
Z=Z_0\psi+K[G],\\
(Z_0\psi)(x,t)=\int_{\mathbb{R}^N}Z(t,x,\xi)\psi(\xi)d\xi,\qquad
K[G](x,t)=\int_0^t\!\!\int_{\mathbb{R}^N}Y(t-\tau,x,\xi)G(\xi,\tau)d\xi d\tau,
\end{gathered}
\end{equation}
is a classical solution of ${}^CD_t z-\mathcal Lz=G(x,t)$, $z(\cdot,0)=\psi$, and $z(x,\cdot)\in H^{\alpha+\lambda}_\nu[0,T]$
for every $\lambda,\nu$ as in \eqref{eq:lamnu} \textup{(\cite[Theorem (b)]{EK})}.\\
(iii) \textup{(}\cite{Ko90}, \cite[Sect.~2.2]{EK}\textup{)} A bounded classical solution of
${}^CD_t z-\mathcal Lz=G$, $z(\cdot,0)=\psi$ with $z(x,\cdot)\in H^{\alpha+\lambda}_\nu[0,T]$, $0<\lambda<1-\alpha$, $0\le\nu<\lambda+1$, is unique.\\
(iv) $Z\ge0$ and $Y\ge0$ \textup{(\cite[Theorem (c)]{EK})}.
\end{theorem}

The structural absence of a zero-order term in the operator $\mathcal{L}$ yields two exact identities.

\begin{lemma}\label{lem:norm}
For all $x\in\mathbb{R}^N$ and $t\in(0,T]$,
\begin{equation}\label{eq:norm}
\qquad\qquad\qquad\qquad \int_{\mathbb{R}^N}Z(t,x,\xi)d\xi=1,\qquad \int_{\mathbb{R}^N}Y(t,x,\xi)d\xi=\frac{t^{\alpha-1}}{\Gamma(\alpha)} .
\end{equation}
\end{lemma}
\textit{Proof.}
Since $\mathcal L1=0$, the function $z\equiv1$ is a bounded classical solution of
${}^CD_t z-\mathcal Lz=0$, $z(\cdot,0)=1$, and $1\in H^{\alpha+\lambda}_\nu$ (because $t^\nu$ is H\"{o}lder of order
$\alpha+\lambda\le\nu$). By (ii)--(iii) of Theorem~\ref{thm:EK}, $1=Z_0 1$, which is the first identity.
Similarly $z=t^{\alpha}/\Gamma(1+\alpha)$ solves ${}^CD_t z-\mathcal Lz=1$, $z(\cdot,0)=0$, and lies in the uniqueness class, so
\[
\qquad\qquad\qquad \frac{t^{\alpha}}{\Gamma(1+\alpha)}=K[1](x,t)=\int_0^t h(s,x) d s,\qquad h(s,x):=\int_{\mathbb{R}^N}Y(s,x,\xi) d\xi.
\]
The function $h(\cdot,x)$ is continuous on $(0,T]$ (dominated convergence, using Lemma~\ref{lem:unified} below), therefore
differentiation gives $h(t,x)=t^{\alpha-1}/\Gamma(\alpha)$.

\subsection*{3.2. Dimension-dependent behaviour near the diagonal $x=\xi$}

Formulas \eqref{eq:Z0}, \eqref{eq:Y0} depend on $x-\xi$ only through
$\rho_\xi(x-\xi)$. Put
\begin{equation*}
 \qquad\qquad\qquad\qquad\qquad\qquad r=\rho_\xi(x-\xi)^{1/2}\,t^{-\alpha/2}
   =\bigl|A(\xi)^{-1/2}(x-\xi)\bigr|\,t^{-\alpha/2}.
\end{equation*}
Then $\rho_\xi(x-\xi)^{-N/2}=t^{-N\alpha/2}r^{-N}$ and the argument of the
$H$-functions equals $r^2/4$, so that
\begin{equation*}
\qquad\qquad\qquad\qquad  Z_0=\frac{t^{-N\alpha/2}}{(\det A(\xi))^{1/2}}\,\Phi_N(r),\qquad
  Y_0=\frac{t^{\alpha-1-N\alpha/2}}{(\det A(\xi))^{1/2}}\,\Psi_N(r),
\end{equation*}
where
\begin{equation*}
\Phi_N(r)=\pi^{-N/2}r^{-N}H^{20}_{12}\!\left[\tfrac{r^2}{4}\,\bigg|\,
  \begin{matrix}(1,\alpha)\\ (\tfrac N2,1),\,(1,1)\end{matrix}\right],\qquad
  \Psi_N(r)=\pi^{-N/2}r^{-N}H^{20}_{12}\!\left[\tfrac{r^2}{4}\,\middle|\,
  \begin{matrix}(\alpha,\alpha)\\ (\frac{N}{2},1),\,(1,1)\end{matrix}\right].
\end{equation*}
 By uniform ellipticity, $r\asymp|x-\xi|\,t^{-\alpha/2}$
and $(\det A(\xi))^{-1/2}\asymp1$ uniformly in $\xi$.
By \cite[(3.7)--(3.9)]{EK}, $\Phi_N$ and $\Psi_N$ decay like
$\exp(-\sigma r^{2/(2-\alpha)})$ as $r\to\infty$ (the power factors are absorbed
by decreasing $\sigma$).
Their behavior as $r\to 0$ is governed by the expansions
\cite[(3.10)--(3.17)]{EK}, and it is \emph{different in every  dimension}.
Table~\ref{tab:sing} lists the orders of the singularities.
\begin{table}[ht]
\centering
\renewcommand{\arraystretch}{1.25}
\begin{tabular}{c|ccc|ccc}
\toprule
$N$ & $\Phi_N$ & $|\nabla\Phi_N|$ & $|D^2\Phi_N|$ & $\Psi_N$ & $|\nabla\Psi_N|$ & $|D^2\Psi_N|$\\
\midrule
$1$ & $O(1)$ & $O(1)$ & $O(1)$ & $O(1)$ & $O(1)$ & $O(1)$\\
$2$ & $|\log r|$ & $r^{-1}$ & $r^{-2}$ & $O(1)$ & $O(1)$ & $|\log r|$\\
$3$ & $r^{-1}$ & $r^{-2}$ & $r^{-3}$ & $O(1)$ & $O(1)$ & $r^{-1}$\\
$4$ & $r^{-2}$ & $r^{-3}$ & $r^{-4}$ & $|\log r|$ & $r^{-1}$ & $r^{-2}$\\
$\ge5$ & $r^{2-N}$ & $r^{1-N}$ & $r^{-N}$ & $r^{4-N}$ & $r^{3-N}$ & $r^{2-N}$\\
\midrule
in $L^1(\mathbb{R}^N)$? & yes & yes & no ($N\ge2$) & yes & yes & yes\\
\bottomrule
\end{tabular}
\caption{Behaviour of the normalised kernels as $r=|x-\xi|t^{-\alpha/2}\to0$
(from \cite{EK}, [Propositions 1, 2 and (3.10), (3.14)]).}
\label{tab:sing}
\end{table}

\noindent We provide some comments on Table~\ref{tab:sing}.
\begin{itemize} 
\item $Y$ is two orders less singular than $Z$; it is bounded near the diagonal only for $N\le3$, which is why
the pointwise bounds \cite[(12)--(15)]{DD21} for $N=1$ have no analogue for $N\ge4$. All arguments below use only integral bounds.
\item $D^2_xZ$ is not integrable for $N\ge2$; this is the reason why H\"{o}lder continuity of the data is required for $N\ge2$
(and not for $N=1$), exactly as in the classical parabolic theory.
\item For $N=4$ the entries for $\Psi_4$ are those given by the leading term $z^2\log z$ of the expansion \cite[(3.14)]{EK}:
$Y_0\approx c\,t^{-\alpha-1}\log(t^{-\alpha}|x|^2)$, hence $|D^m_xY_0|\asymp t^{-\alpha-1}|x|^{-| m|}$, $|m|\ge1$.
(The form of \cite[(2.14)--(2.16)]{EK} for $n=4$ is different; only the integrability recorded in the last row matters
below, and it holds in either version.)
\end{itemize}

The corrections $V_Z,V_Y$ obey the dimension-dependent estimates \cite[Theorem (a)]{EK}. Writing each of them in terms of $r$,
one finds that they are of the same form as the bounds for $Z_0,Y_0$, multiplied by an additional factor $t^{\delta_N}$,
$\delta_N>0$; Table~\ref{tab:V} records $\delta_N$ (here $0<\gamma_0<\gamma$ is the auxiliary exponent of \cite{EK}, with
$\gamma_0<\gamma/2$ when $N=4$).

\begin{table}[ht]
\centering
\renewcommand{\arraystretch}{1.25}
\begin{tabular}{c|ccccc}
\toprule
$N$ & $1$ & $2$ & $3$ & $4$ & $\ge5$\\
\midrule
$\delta_N$ for $V_Z$, $|m|\le1$ & $\gamma\alpha/2$ & $\gamma\alpha/2$ & $\gamma\alpha/2$ & $\alpha(2\gamma-3\gamma_0)/2$ & $\gamma\alpha/2$\\
$\delta_N$ for $V_Y$, $|m|\le1$ & $\gamma\alpha/2$ & $\gamma\alpha/2$ & $\gamma\alpha/2$ & $\alpha(3\gamma-5\gamma_0)/4$ & $\gamma\alpha/2$\\
\bottomrule
\end{tabular}
\caption{The gain of the corrections $V_Z$, $V_Y$ over the homogeneous orders $t^{-(N+|m|)\alpha/2}$ and
$t^{\alpha-1-(N+|m|)\alpha/2}$.}
\label{tab:V}
\end{table}

For instance, when  $N\ge5$, \cite[Theorem (a)(v)]{EK} yields
\[
|D^m_xV_Y|\le Ct^{(\gamma_0+\gamma)\alpha/4-1}|x-\xi|^{2-N-|m|+(\gamma-\gamma_0)/2}e^{-\sigma\rho};
\]
substituting $|x-\xi|=t^{\alpha/2}r$ yields the bound
$Ct^{\alpha-1-(N+|m|)\alpha/2}\cdot t^{\gamma\alpha/2}\cdot r^{2-N-|m|+(\gamma-\gamma_0)/2}e^{-\sigma\rho}$.
The other entries are computed in the same way; for $N=4$ the positivity of $\delta_4$ uses $\gamma_0<\gamma/2$.

Since $t^{\delta_N}\le T^{\delta_N}$, Tables~\ref{tab:sing}--\ref{tab:V} can be consolidated into a single statement that holds in all dimensions.
Put
\begin{equation}\label{eq:omega}
\omega(r)=\bigl(1+r^{1-N}\bigr)\exp\bigl(-\sigma r^{2/(2-\alpha)}\bigr),\qquad r>0 .
\end{equation}
Logarithms are absorbed using $|\log r|\le C(1+r^{-1})$ for $r\le1$.

\begin{lemma}[Unified pointwise bounds]\label{lem:unified}
There are $C,\sigma>0$, depending only on $N,\alpha,T$ and the constants in (A1), such that for $0<t\le T$, $x\ne\xi$,
$|m|\le1$,
\begin{gather*}
|D^m_xZ(t,x,\xi)|\le C\,t^{-(N+|m|)\alpha/2}\,\omega\bigl(|x-\xi|t^{-\alpha/2}\bigr),\\
|D^m_xY(t,x,\xi)|\le C\,t^{\alpha-1-(N+|m|)\alpha/2}\,\omega\bigl(|x-\xi|t^{-\alpha/2}\bigr).
\end{gather*}
\end{lemma}
\textit{Proof.}
For $r\ge1$ this is \cite[(2.5),(2.11)]{EK} together with Table~\ref{tab:V}. For $r\le1$ every singular entry of
Table~\ref{tab:sing} in the columns $\Phi_N,\nabla\Phi_N,\Psi_N,\nabla\Psi_N$ is $O(1+r^{1-N})$, and the same holds for the powers of $r$ in the
estimates of $V_Z,V_Y$ (they are of order $\ge 1-N$); on $r\le1$ the exponential factor is bounded below by $e^{-\sigma}$.

Because $\int_{\mathbb{R}^N}\omega(|y|)|y|^{\beta} d y<\infty$ for $\beta>-1$, the substitution $\xi=x+t^{\alpha/2}y$ gives:

\begin{lemma}[Integral bounds]\label{lem:int}
For $\beta\in[0,1]$, $|m|\le1$, $0<t\le T$ and every $x\in\mathbb{R}^N$,
\[
\int_{\mathbb{R}^N}|D^m_xZ(t,x,\xi)|\,|x-\xi|^{\beta} d\xi\le C\,t^{\alpha(\beta-|m|)/2},\qquad
\int_{\mathbb{R}^N}|D^m_xY(t,x,\xi)|\,|x-\xi|^{\beta} d\xi\le C\,t^{\alpha-1+\alpha(\beta-|m|)/2}.
\]
\end{lemma}

These estimates are \emph{the same in all dimensions}; dimension enters only through the constants. Note that for $|m|=1$,
$\beta=0$ the $Y$-bound is $Ct^{\alpha/2-1}$, which is integrable at $t=0$, while $\int|D^2_xY| d\xi\sim t^{-1}$ is not.

\subsection*{3.3. The operators $Z_0$ and $K$}

\begin{lemma}\label{lem:ops}
Let $\ell\in(0,1)$. There exist  constants $c_1,c_2,c_3,c_4$ (depending on $N,\alpha,\ell,T$ and (A1) such that: \\
(a)$|Z_0\psi(x,t)|\le|\psi|_0$ and $[Z_0\psi(\cdot,t)]_{\ell}\le c_1[\psi]_{\ell}$ for $\psi\in H^{\ell}(\mathbb{R}^N)$, $t\in(0,T]$;\\
(b) for bounded continuous $G$, $|K[G](x,t)|\le(I_{0+}^{\alpha} \|G\|)(t)$ and
\[
\qquad\qquad[K[G](\cdot,t)]_{\ell}\le c_2\int_0^t(t-\tau)^{\alpha(1-\ell/2)-1}\|G\|(\tau) d\tau\le
\frac{c_2\,t^{\alpha(1-\ell/2)}}{\alpha(1-\ell/2)}\sup_{\tau\le t} \|G\|(\tau);
\]\\
(c) if $\psi$ is Lipschitz with constant $\Lambda$, then $Z_0\psi(\cdot,t)\in C^1$ and $|\nabla_xZ_0\psi|\le c_3\Lambda$;
for bounded continuous $G$, $K[G](\cdot,t)\in C^1$ and $|\nabla_x K[G](x,t)|\le c_4\int_0^t(t-\tau)^{\alpha/2-1} \|G\|(\tau) d\tau$. 
\end{lemma}

\textit{Proof.}  (a) The first bound follows from $Z\ge0$ and \eqref{eq:norm}. Let $h=|x-\tilde x|>0$. If $h\ge t^{\alpha/2}$, write
$Z_0\psi(x)-Z_0\psi(\tilde x)=\int Z(t,x,\xi)(\psi(\xi)-\psi(x)) d\xi-\int Z(t,\tilde x,\xi)(\psi(\xi)-\psi(\tilde x)) d\xi+\psi(x)-\psi(\tilde x)$
(by \eqref{eq:norm}); Lemma~\ref{lem:int} with $m=0$, $\beta=\ell$ bounds this by $[\psi]_{\ell}(2Ct^{\alpha\ell/2}+h^{\ell})\le(2C+1) [\psi]_{\ell} h^{\ell}$.
If $h<t^{\alpha/2}$, again by \eqref{eq:norm},
\[
\qquad\qquad\qquad Z_0\psi(x)-Z_0\psi(\tilde x)=\int\bigl[Z(t,x,\xi)-Z(t,\tilde x,\xi)\bigr]\bigl[\psi(\xi)-\psi(x)\bigr] d\xi .
\]
For almost every $\xi$ the segment $x_\theta=x+\theta(\tilde x-x)$, $\theta\in[0,1]$, avoids $\xi$, so
$Z(t,x,\xi)-Z(t,\tilde x,\xi)=-\int_0^1\nabla_xZ(t,x_\theta,\xi)\cdot(\tilde x-x) d\theta$. By
$|\psi(\xi)-\psi(x)|\le [\psi]_{\ell}(|\xi-x_\theta|^\ell+h^\ell)$, Fubini's theorem and Lemma~\ref{lem:int} ($m=1$, $\beta=\ell$ and $\beta=0$), we have
\[
|Z_0\psi(x)-Z_0\psi(\tilde x)|\le C[\psi]_{\ell}\,h\bigl(t^{\alpha(\ell-1)/2}+h^{\ell}t^{-\alpha/2}\bigr)
=C[\psi]_{\ell}\,h^{\ell}\bigl((ht^{-\alpha/2})^{1-\ell}+ht^{-\alpha/2}\bigr)\le2C[\psi]_{\ell} h^{\ell}.
\]
(b) The first bound follows from $Y\ge0$ and \eqref{eq:norm}. For $s>0$, by the same segment argument and Lemma~\ref{lem:int},
$\int|Y(s,x,\xi)-Y(s,\tilde x,\xi)| d\xi\le\min\{2s^{\alpha-1}/\Gamma(\alpha),\ Chs^{\alpha/2-1}\}\le c\,h^{\ell}s^{\alpha-1-\alpha\ell/2}$
(interpolate $\min\{a,b\}\le a^{1-\ell}b^{\ell}$). Integrating against $\|G\|(\tau)$ with $s=t-\tau$ gives the claim.

(c) By Lemma~\ref{lem:int} the $x$-differentiated integrals converge uniformly on compact sets, so differentiation under the
integral sign is legitimate. Differentiating the first identity in \eqref{eq:norm} gives $\int\nabla_xZ(t,x,\xi) d\xi=0$, hence
$\nabla_xZ_0\psi(x,t)=\int\nabla_xZ(t,x,\xi)[\psi(\xi)-\psi(x)] d\xi$, which is bounded by
$\Lambda\int|\nabla_xZ||\xi-x| d\xi\le C\Lambda$. The bound for $\nabla_x K[G]$ is verified by  Lemma~\ref{lem:int} with $m=1,\beta=0$.

We will make repeated use of the following fractional Gronwall inequality.

\begin{lemma}\label{lem:gronwall}
Let $\mu\ge0$ be bounded and measurable on $[0,T]$, $A,B\ge0$, and $\mu(t)\le A+B(I_{0+}^{\alpha}\mu)(t)$ for $t\in[0,T]$. Then
$\mu(t)\le A\,E_{\alpha}(Bt^{\alpha})$, where $E_{\alpha}(s)=\sum_{k\ge0}s^k/\Gamma(1+k\alpha)$ is the Mittag-Leffler function \cite[pp. 40--45]{KST}.
\end{lemma}
\textit{Proof.}
Since $I_{0+}^{\alpha}$ is positive and $I^{k\alpha}_{0+}1=t^{k\alpha}/\Gamma(1+k\alpha)$, iterating $K$ times gives
\[\qquad\qquad\qquad\qquad\qquad
\mu(t)\le A\sum_{k<K}\frac{B^kt^{k\alpha}}{\Gamma(1+k\alpha)}+B^K\sup\mu\,\frac{t^{K\alpha}}{\Gamma(1+K\alpha)},
\]
and the last term tends to $0$ as $K\to\infty$.

\section*{4. The direct problem}\label{sec:direct}

Fix $N\ge1$, an operator $\mathcal L$ satisfying (A1) in $\mathbb{R}^N$ and $\ell\in(0,1)$. By Theorem~\ref{thm:EK}(ii), any solution of
\eqref{eq:P} must satisfy the Volterra-type integral equation obtained by transferring the term $pz$ to the right-hand side:
\begin{equation}\label{eq:IE}
\qquad\qquad\qquad\qquad\qquad z=z_0-K[pz],\qquad z_0:=Z_0\psi+K[F].
\end{equation}
Put
\[
\qquad\qquad\qquad P:=\|p\|_0,\qquad \lambda_0:=|\psi|_0+\frac{T^{\alpha}}{\Gamma(1+\alpha)} \|F\|_0,\qquad M:=\lambda_0\,E_{\alpha}(PT^{\alpha}).
\]

\begin{theorem}[Solvability of the direct problem]\label{thm:direct}
Let $\psi\in H^{\ell}(\mathbb{R}^N)$, $F\in C^{\ell,0}_b(\mathbb{R}^N\times[0,T])$, $p\in C^{\ell,0}_b(\mathbb{R}^N\times[0,T])$. Then:\\
(i) The integral equation \eqref{eq:IE} has a unique solution $z\in C_b(\mathbb{R}^N\times[0,T])$. It is the sum of the uniformly convergent series
$z=\sum_{k\ge0}z_k$, $z_k=-K[pz_{k-1}]$, and
\begin{equation}\label{eq:supbound}
\qquad\qquad\qquad\qquad |z_k(x,t)|\le\lambda_0\frac{P^k t^{k\alpha}}{\Gamma(1+k\alpha)},\qquad
|z(x,t)|\le\lambda_0\,E_{\alpha}(Pt^{\alpha})\le M .
\end{equation}\\
(ii) $z\in C^{\ell,0}_b(\mathbb{R}^N\times[0,T])$ and
\begin{equation}\label{eq:holderbound}
\qquad\qquad\qquad\qquad \sup_{t\le T}[z(\cdot,t)]_{\ell}\le c_1 [\psi]_{\ell}+\frac{c_2T^{\alpha(1-\ell/2)}}{\alpha(1-\ell/2)}\bigl(\|F\|_0+PM\bigr).
\end{equation}\\
(iii) $z$ is a unique regular solution of \eqref{eq:P}.\\
(iv) If, in addition, $\psi$ is Lipschitz with constant $\Lambda$, then $\nabla_xz$ is bounded:
$|\nabla_xz|\le c_3\Lambda+\frac{2c_4}{\alpha}T^{\alpha/2}(\|F\|_0+PM)$. 
\end{theorem}

\textit{Proof.} (i) By Lemma~\ref{lem:ops}(a),(b), $|z_0|\le|\psi|_0+(I_{0+}^{\alpha}\|F\|)(t)\le\lambda_0$, and
$\|z_k\|(t)\le P(I_{0+}^{\alpha}\|z_{k-1}\|)(t)$. Induction with $I_{0+}^{\alpha} t^{k\alpha}=\frac{\Gamma(1+k\alpha)}{\Gamma(1+(k+1)\alpha)}t^{(k+1)\alpha}$
gives the first bound in \eqref{eq:supbound}. Hence the series is majorised by $\lambda_0\sum_k(PT^\alpha)^k/\Gamma(1+k\alpha)=M$
and converges uniformly; each $z_k$ is bounded and continuous ($K$ maps $C_b(\mathbb{R}^N\times[0,T])$ into itself: continuity for
$t>0$ follows by dominated convergence with the majorant of Lemma~\ref{lem:unified}, and $|K[G]|\le \|G\|_0t^\alpha/\Gamma(1+\alpha)$
gives continuity at $t=0$; $Z_0\psi$ is continuous by Theorem~\ref{thm:EK}(ii)). Summing the recursion shows that $z$ solves
\eqref{eq:IE}. If $d\in C_b$ solves $d=-K[pd]$, then $\|d\|(t)\le P(I_{0+}^{\alpha} \|d\|)(t)$, and Lemma~\ref{lem:gronwall} with $A=0$
yields $d=0$.

(ii) Apply Lemma~\ref{lem:ops}(a),(b) to \eqref{eq:IE}, $z=Z_0\psi+K[F-pz]$, and use $\|F-pz\|(\tau)\le\|F\|_0+PM$.

(iii) Put $G=F-pz$. By (ii), $G$ is bounded, continuous and H\"{o}lder continuous in $x$ uniformly in $t$
($[pz]_{\ell}\le | p_0|[z]_{\ell}+[p]_{\ell}|z|_0$). Since $z=Z_0\psi+K[G]$, Theorem~\ref{thm:EK}(ii) shows that $z$ is a classical
solution of ${}^CD_t z-\mathcal Lz=G=F-pz$, $z(\cdot,0)=\psi$, with $z(x,\cdot)\in H^{\alpha+\lambda}_\nu$; together with (ii),
$z\in \mathcal{U}_T(\mathbb{R}^N)$ is a regular solution of \eqref{eq:P}.

Let $\tilde z$ be another regular solution. Then $\tilde G=F-p\tilde z$ is bounded, continuous and H\"{o}lder in $x$ uniformly in $t$
(Definition~\ref{def:reg}(a)), so $Z_0\psi+K[\tilde G]$ is a classical solution of ${}^CD_t\cdot-\mathcal L\cdot=\tilde G$ with
initial value $\psi$ in the class $H^{\alpha+\lambda}_\nu$, and so is $\tilde z$. By Theorem~\ref{thm:EK}(iii),
$\tilde z=Z_0\psi+K[F]-K[p\tilde z]$, i.e.\ $\tilde z$ solves \eqref{eq:IE}, and $\tilde z=z$ by (i).

(iv) Lemma~\ref{lem:ops}(c) applied to $z=Z_0\psi+K[F-pz]$.

\begin{remark}
For $N=1$ the H\"{o}lder continuity of $\psi,F,p$ in $x$ is needed only for (ii); existence, uniqueness and \eqref{eq:supbound} hold for
bounded continuous data. This is consistent with \cite{DD21}, where $q=q(t)$ is only continuous. For $N\ge2$ H\"{o}lder continuity is
indispensable, because $D^2_xZ$ is not integrable (Table~\ref{tab:sing}).
\end{remark}

\begin{theorem}[Stability for the direct problem]\label{thm:stab}
Let $z$ and $\tilde z$ be the regular solutions of \eqref{eq:P} with data $(\psi,F,p)$ and $(\tilde\psi,\tilde F,\tilde p)$ satisfying the
hypotheses of Theorem~\ref{thm:direct}, and let $M$ be the constant of $z$ as in (\ref{eq:supbound}), $\tilde P=\|\tilde p\|_0$. Then for $t\in[0,T],$ we have
\begin{equation}\label{eq:stab0}
 \qquad\qquad\qquad \|z-\tilde z\|(t) \le\Bigl[|\psi-\tilde\psi|_0+\frac{T^{\alpha}}{\Gamma(1+\alpha)}\bigl(\|F-\tilde F\|_0+M\|p-\tilde p\|_0\bigr)\Bigr]E_{\alpha}(\tilde Pt^{\alpha}),   
\end{equation}
\begin{align}
\qquad\qquad [(z-\tilde z)(\cdot,t)]_{\ell}&\le c_1 [\psi-\tilde\psi]_{\ell}+\frac{c_2T^{\alpha(1-\ell/2)}}{\alpha(1-\ell/2)}\bigl(\|F-\tilde F\|_0+M\|p-\tilde p\|_0+\tilde P\|z-\tilde z\|_0\bigr).\label{eq:stabl}
\end{align}
Consequently $[z-\tilde z]_\ell\le C_s\bigl(|\psi-\tilde\psi|_\ell+\|F-\tilde F\|_0+\|p-\tilde p\|_0\bigr)$ with $C_s$ depending only on
$N,\alpha,\ell,T$, (A1), $\lambda_0$, $P$ and $\tilde P$. In particular the regular solution depends continuously on the data, and on $p$
only through its \emph{sup-norm}.
\end{theorem}
\textit{Proof.} Subtracting the integral equations, $d=z-\tilde z$ satisfies
$d=Z_0(\psi-\tilde\psi)+K[F-\tilde F]-K[(p-\tilde p)z]-K[\tilde pd]$. By Lemma~\ref{lem:ops} (a),(b) and \eqref{eq:supbound},
$\|d\|(t)\le A+\tilde P(I_{0+}^{\alpha} \|d\|)(t)$ with $A$ the bracket in \eqref{eq:stab0}; Lemma~\ref{lem:gronwall} gives \eqref{eq:stab0}.
The H\"{o}lder  part follows from Lemma~\ref{lem:ops} (a),(b) applied to the same identity.

\begin{remark}\label{rem:compare}
For $\mathcal L=\Delta$, $N=1,2,3$, the estimate \eqref{eq:supbound} is \cite[(18)]{DD21}, \cite[(21)]{DRB21}, \cite[(18)]{UD21} with
$C_0\Gamma(\alpha)=1.$
\end{remark}

\section*{5. Reduction of the inverse problem: the equivalent problem}\label{sec:reduction}

In what follows, $N=n$ and $\mathcal L=L$ satisfies (A1)--(A2), and the coefficient is $p(x,t)=q(x',t)$. We assume \\
(D1) $\varphi\in H^{2+\ell}(\mathbb{R}^n)$;\\
(D2) $f,\ \partial_{x_n}f,\ \partial^2_{x_n}f\in C^{\ell,0}_b(\overline\Pi_T)$;\\
(D3) $g$ and $\partial^{\beta}_{x'}g$, $|\beta|\le2$, belong to $C^{\ell,0}_b(\mathbb{R}^{n-1}\times[0,T])$, and for every $x'$ the function $g(x',\cdot)$ satisfies (c), (d) of Definition~\ref{def:reg} (so ${}^CD_t g$ is defined on $(0,T]$). For $n=1$: $g\in C[0,T]\cap H^{\alpha+\lambda}_\nu[0,T]$ and
$I^{1-\alpha}_{0+}[g-g(0)]\in C^1(0,T]$;\\
(D4) $q\in C^{\ell,0}_b(\mathbb{R}^{n-1}\times[0,T])$ (for $n=1$: $q\in C[0,T]$).

Applying Theorem~\ref{thm:direct} with $\psi=\varphi$, $F=f$, we conclude that problem \eqref{eq:main}--\eqref{eq:init} admits a unique regular solution  $u$.
Consider also the two direct problems
\begin{align}
& \qquad\qquad\qquad {}^CD_t w-Lw+q(x',t)w=\partial_{x_n}f,\qquad w(x,0)=\partial_{x_n}\varphi(x),\label{eq:Pw}\tag{$\mathrm P_w$}\\
&\qquad\qquad\qquad {}^CD_t v-Lv+q(x',t)v=\partial^2_{x_n}f,\qquad v(x,0)=\partial^2_{x_n}\varphi(x),\label{eq:Pv}\tag{$\mathrm P_v$}
\end{align}
which by (D1), (D2) and Theorem~\ref{thm:direct} have unique regular solutions; moreover $w$ has a bounded gradient
(Theorem~\ref{thm:direct}(iv), since $\partial_{x_n}\varphi$ is Lipschitz).

\subsection*{5.1. Differentiation with respect to $x_n$}

\begin{lemma}\label{lem:transl}
Under (A2), if $z$ is a regular solution of \eqref{eq:P} with $\mathcal L=L$, coefficient $q(x',t)$ and data $(\psi,F)$, then for
$h\in \mathbb{R}$ the function $z(x+he_n,t)$ is the regular solution with the same coefficients and the data
$(\psi(\cdot+he_n),F(\cdot+he_n,\cdot))$.
\end{lemma}
\textit{Proof.} The coefficients of $L$ and $q$ are invariant under $x\mapsto x+he_n$; all conditions of Definition~\ref{def:reg} are translation invariant.

\begin{theorem}\label{thm:diff}
Assume (A1)--(A2), (D1), (D2), (D4)). Then $\partial_{x_n}u$ and $\partial^2_{x_n}u$ exist and are continuous in $\overline\Pi_T$, and
\begin{equation*}
\qquad\qquad\qquad \partial_{x_n}u=w,\qquad \partial^2_{x_n}u=v,\qquad \partial_{x_i}\partial_{x_n}u=\partial_{x_i}w\quad \text{for} \, \, t\in(0,T],\ 1\le i\le n.
\end{equation*}
\end{theorem}
\textit{Proof.}
For $h\ne0$ let $\Delta_hz(x,t)=h^{-1}[z(x+he_n,t)-z(x,t)]$. By Lemma~\ref{lem:transl} and linearity, $\Delta_hu$ is the regular solution
of \eqref{eq:P} with coefficient $q$ and data $(\Delta_h\varphi,\Delta_hf)$, which satisfy the hypotheses of Theorem~\ref{thm:direct}.
Theorem~\ref{thm:stab} with $p=\tilde p=q$ gives
\[
\qquad\qquad\qquad \sup_{\overline\Pi_T}|\Delta_hu-w|\le\Bigl[|\Delta_h\varphi-\partial_{x_n}\varphi|_0+\frac{T^\alpha}{\Gamma(1+\alpha)}\|\Delta_hf-\partial_{x_n}f\|_0\Bigr]
E_{\alpha}(\|q\|_0T^{\alpha}).
\]
By the mean value theorem, $|\Delta_h\varphi-\partial_{x_n}\varphi|_0\le [\partial_{x_n}\varphi]_{\ell}|h|^\ell$ and
$\|\Delta_hf-\partial_{x_n}f\|_0\le\sup_t[\partial_{x_n}f(\cdot,t)]_{\ell} |h|^\ell$, so $\Delta_hu\to w$ uniformly, i.e.\ $\partial_{x_n}u$ exists
and equals $w$. By the same argument for $w$ (data $\partial_{x_n}\varphi$, $\partial_{x_n}f$, whose $x_n$-derivatives $\partial^2_{x_n}\varphi$,
$\partial^2_{x_n}f$ are H\"{o}lder continuous by (D1), (D2) gives $\partial_{x_n}w=v$. The last identity holds because $u(\cdot,t)\in C^2$ for $t>0$.

\subsection*{5.2. The trace condition and the equivalence theorem}
Evaluating \eqref{eq:main} on $\{x_n=0\}$ and splitting $L$ into the tangential part \eqref{eq:Lprime} and the terms containing $\partial_{x_n}$, we get
\[
\qquad\qquad\qquad Lu\big|_{x_n=0}=L'\bigl(u|_{x_n=0}\bigr)+\Bigl[a_{nn}\,\partial^2_{x_n}u+2\sum_{i=1}^{n-1}a_{in}\,\partial_{x_i}\partial_{x_n}u+b_n\,\partial_{x_n}u\Bigr]_{x_n=0}.
\]
This motivates the following trace condition on the pair $(w,v)$:
\begin{equation*}
\qquad a_{nn}(x')v(x',0,t)+2\sum_{i=1}^{n-1}a_{in}(x')\,\partial_{x_i}w(x',0,t)+b_n(x')\,w(x',0,t)
\end{equation*}
\begin{equation}\label{eq:trace}
\qquad\qquad\qquad\qquad\qquad={}^CD_t g(x',t)-L'g(x',t)+q(x',t)g(x',t)-f(x',0,t)    
\end{equation}
for $x'\in \mathbb{R}^{n-1}$, $t\in(0,T]$ (for $n=1$: $a\,v(0,t)+b\,w(0,t)={}^CD_t g(t)+q(t)g(t)-f(0,t)$).

\begin{theorem}[Equivalence]\label{thm:equiv}
Assume (A1)--(A2), (D1)--(D3), (D4). The following statements are equivalent:\\
(I) the regular solution $u$ of \eqref{eq:main}, \eqref{eq:init} satisfies the overdetermination condition \eqref{eq:over};\\
(II) the compatibility condition $\varphi(x',0)=g(x',0)$ holds, and the regular solutions $w$, $v$ of problems \eqref{eq:Pw}, \eqref{eq:Pv} satisfy
the trace condition \eqref{eq:trace}. \\
Thus $q$ solves the inverse problem \eqref{eq:main}--\eqref{eq:over} if and only if $q$, together with the regular solutions of the direct problems
\eqref{eq:Pw}, \eqref{eq:Pv}, satisfies \eqref{eq:trace}, provided $\varphi(x',0)=g(x',0)$.
\end{theorem}
\textit{Proof.} Show (I)$\Rightarrow$(II). Since $u$ is continuous up to $t=0$, \eqref{eq:over} and \eqref{eq:init} give $g(x',0)=u(x',0,0)=\varphi(x',0)$. Setting
$x_n=0$ in \eqref{eq:main}, using $u(x',0,t)=g(x',t)$ (so that ${}^CD_t$ and $L'$ of $u|_{x_n=0}$ are those of $g$) and Theorem~\ref{thm:diff},
we obtain \eqref{eq:trace}.

(II)$\Rightarrow$(I). Put $\tilde u(x',t)=u(x',0,t)$ and $\vartheta=\tilde u-g$. By Theorem~\ref{thm:diff}, \eqref{eq:main} at $x_n=0$ reads
\[
\qquad\qquad\qquad\qquad {}^CD_t\tilde u-L'\tilde u-\Bigl[a_{nn}v+2\sum_{i<n}a_{in}\partial_{x_i}w+b_nw\Bigr]_{x_n=0}+q\tilde u=f(x',0,t).
\]
Subtracting \eqref{eq:trace} we obtain
\begin{equation}\label{eq:theta}
\qquad\qquad{}^CD_t\vartheta-L'\vartheta+q\vartheta=0\quad\text{in }\mathbb{R}^{n-1}\times(0,T],\qquad \vartheta(x',0)=\varphi(x',0)-g(x',0)=0 .
\end{equation}
Let $n\ge2$. The function $\vartheta$ belongs to $\mathcal{U}_T(\mathbb{R}^{n-1})$: it is bounded and continuous, H\"{o}lder in $x'$ uniformly in $t$
($u\in C^{\ell,0}_b$ by Theorem~\ref{thm:direct}(ii), $g$ by (D3), twice continuously differentiable in $x'$ for $t>0$, and (c), (d) of
Definition~\ref{def:reg} hold for $u(x',0,\cdot)$ and, by (D3), for $g(x',\cdot)$, hence for the difference. Since $L'$ satisfies (A1) in
$\mathbb{R}^{n-1}$ and $q\in C^{\ell,0}_b(\mathbb{R}^{n-1}\times[0,T])$, $\vartheta$ is a regular solution of \eqref{eq:P} in $\mathbb{R}^{n-1}$ with $\psi=0$, $F=0$; by
the uniqueness part of Theorem~\ref{thm:direct}, $\vartheta\equiv0$, which is \eqref{eq:over}.

Let $n=1$; then \eqref{eq:theta} is ${}^CD_t\vartheta+q\vartheta=0$, $\vartheta(0)=0$. The function $\Theta=I^{1-\alpha}_{0+}\vartheta$ is continuous on
$[0,T]$, $\Theta(0)=0$, and $\Theta'=-q\vartheta$ is bounded and continuous on $(0,T]$; hence $\Theta=I^{1}_{0+}(-q\vartheta)=I^{1-\alpha}_{0+}I_{0+}^{\alpha}(-q\vartheta)$,
and the injectivity of $I^{1-\alpha}_{0+}$ on $C[0,T]$ gives $\vartheta=-I_{0+}^{\alpha}(q\vartheta)$. Lemma~\ref{lem:gronwall} with $A=0$ yields $\vartheta=0$.

\begin{corollary}[The case (A3)]\label{cor:A3}
If in addition $a_{in}\equiv0$ \textup{(}$i<n$\textup{)} and $b_n\equiv0$, the function $w$ does not enter \eqref{eq:trace}, and the inverse problem is
equivalent to the direct problem \eqref{eq:Pv} together with
\begin{equation}\label{eq:traceA3}
a_{nn}(x')\,v(x',0,t)={}^CD_t g(x',t)-L'g(x',t)+q(x',t)g(x',t)-f(x',0,t),\qquad \varphi(x',0)=g(x',0).
\end{equation}
For $L=\Delta$ ($a_{nn}=1$, $L'=\Delta_{x'}$) this is exactly the equivalent problem of \cite[(5)--(7)]{DD21} for $n=1$,
\cite[(4)--(6)]{DRB21} for $n=2$ and \cite[(4)--(6)]{UD21} for $n=3$.
\end{corollary}

\subsection*{5.3. The equivalent system of integral equations}
Assume, as in \cite{DD21,DRB21,UD21}, that $|g(x',t)|\ge g_0>0$. By Theorem~\ref{thm:direct}, the regular solutions of \eqref{eq:Pw},
\eqref{eq:Pv} are the solutions of
\begin{align}
\qquad\qquad\qquad\qquad\qquad w&=w_0-K[qw],\qquad w_0=Z_0\,\partial_{x_n}\varphi+K[\partial_{x_n}f],\label{eq:IEw}\\
\qquad\qquad\qquad\qquad\qquad v&=v_0-K[qv],\qquad\ \ v_0=Z_0\,\partial^2_{x_n}\varphi+K[\partial^2_{x_n}f],\label{eq:IEv}
\end{align}
where $Z,Y$ are the  kernels of $L$ in $\mathbb{R}^n$ from  Theorem \ref{thm:EK} and
$K[G](x,t)=\int_0^t\!\int_{\mathbb{R}^n}Y(t-\tau,x,\xi)G(\xi,\tau) d\xi d\tau$. By Lemma~\ref{lem:ops}(c) the tangential derivatives in \eqref{eq:trace}
may be computed under the integral sign. Dividing \eqref{eq:trace} by $g$ we obtain the equation for $q$:
\begin{equation}\label{eq:IEq}
\qquad\qquad\qquad\qquad\qquad q(x',t)=q_0(x',t)-\frac{1}{g(x',t)}\,\mathcal B[q](x',t),
\end{equation}
\begin{align*}
\qquad q_0&=\frac{1}{g}\Bigl\{f(x',0,t)-{}^CD_t g+L'g+\Bigl[a_{nn}v_0+2\sum_{i<n}a_{in}\partial_{x_i}w_0+b_nw_0\Bigr]_{x_n=0}\Bigr\},\\
\mathcal \qquad\qquad B[q]&=\Bigl[a_{nn}K[qv]+2\sum_{i<n}a_{in}\partial_{x_i}K[qw]+b_nK[qw]\Bigr]_{x_n=0},
\end{align*}
where $w=w(q)$, $v=v(q)$ are the solutions of \eqref{eq:IEw}, \eqref{eq:IEv}; explicitly, for instance,
\begin{equation*}
\qquad\qquad\qquad\qquad K[qv](x',0,t)=\int_0^t\!\!\int_{\mathbb{R}^n}Y\bigl(t-\tau,(x',0),\xi\bigr)\,q(\xi',\tau)\,v(\xi,\tau) d\xi d\tau .
\end{equation*}

\begin{proposition}\label{prop:system}
According to  the assumptions of Theorem~\ref{thm:equiv} and $|g|\ge g_0>0$, a function $q\in C^{\ell,0}_b(\mathbb{R}^{n-1}\times[0,T])$ is a solution of the inverse
problem \eqref{eq:main}--\eqref{eq:over} if and only if $\varphi(x',0)=g(x',0)$ and the triple $(q,w,v)$ solves the system
\eqref{eq:IEw}--\eqref{eq:IEq}. Under (A3) the system reduces to \eqref{eq:IEv} and
\begin{equation}\label{eq:IEqA3}
\qquad\qquad\qquad q=\frac{1}{g}\Bigl\{f(x',0,t)-{}^CD_t g+L'g+a_{nn}v_0(x',0,t)\Bigr\}-\frac{a_{nn}(x')}{g(x',t)}\,K[qv](x',0,t),
\end{equation}
which for $L=\Delta$ is \cite[(24)]{DD21}, \cite[(27)]{DRB21}, \cite[(23)]{UD21}.
\end{proposition}
\textit{Proof.} The result follows from Theorem~\ref{thm:equiv} and the equivalence, established in Theorem~\ref{thm:direct}, between regular solutions of \eqref{eq:Pw} and \eqref{eq:Pv} and solutions of \eqref{eq:IEw} and \eqref{eq:IEv}.

\section*{6. The inverse problem: setting and main results} 

\subsection*{6.1. Setting}
Throughout Sections 6--12 we are in the situation of Section 5: $N=n$,
$\mathcal L=L$ satisfies (A1)--(A2), $Z,Y$ are the  kernels of $L$ in $\mathbb{R}^n$ defined by Theorem \ref{thm:EK}, and $Z_0$, $K$ are the
operators in \eqref{eq:repr}. For $n=1$ the variable $x'$ is absent, $C^{\ell,0}_b(\mathbb{R}^{0}\times[0,T])=C[0,T]$, and $L'=0$.
Recall the {\it inverse problem} consists in finding $q(x', t)$ from \eqref{eq:main}--\eqref{eq:over}.

\begin{definition}\label{def:sol}
Let $0<T_1\le T$. A function $q\in C^{\ell,0}_b(\mathbb{R}^{n-1}\times[0,T_1])$ is a \emph{solution of the inverse problem on $[0,T_1]$} if the regular
solution $u$ of \eqref{eq:main}--\eqref{eq:init} in $\mathbb{R}^n\times(0,T_1]$ (which exists and is unique by Theorem~\ref{thm:direct}) satisfies \eqref{eq:over}
for $t\in(0,T_1]$.
\end{definition}

\begin{assumption}[Standing assumptions]\label{ass:S}
(A1)--(A2) of Assumption~\ref{ass:A} hold; the data satisfy (D1)--(D3) of Section 5; $\varphi(x',0)=g(x',0)$; and
\begin{equation}\label{eq:g0}
\qquad\qquad\qquad\qquad\qquad |g(x',t)|\ge g_0>0\qquad\text{on } \mathbb{R}^{n-1}\times[0,T].
\end{equation}
\end{assumption}

\subsection*{6.2. The key quantities}
Given the data  $(\varphi,f,g)$ define, for $(x',t)\in \mathbb{R}^{n-1}\times(0,T]$,
\begin{equation}\label{eq:calG}
\qquad\qquad\qquad\qquad \mathcal G(x',t):={}^CD_t g(x',t)-L'g(x',t)-f(x',0,t).
\end{equation}
By Theorem~\ref{thm:equiv} and \eqref{eq:trace}, $q$ solves the inverse problem if and only if
\begin{equation}\label{eq:traceII}
\qquad\qquad\qquad q\,g=S[q]-\mathcal G,\qquad S[q]:=\Bigl[a_{nn}v+2\sum_{i=1}^{n-1}a_{in}\partial_{x_i}w+b_nw\Bigr]_{x_n=0},
\end{equation}
where $w=w(q)$, $v=v(q)$ are the regular solutions of \eqref{eq:Pw}, \eqref{eq:Pv}. Since $S[q]$, $q$ and $g$ are bounded,
\eqref{eq:traceII} shows that $\mathcal G$ is bounded on $\mathbb{R}^{n-1}\times(0,T]$ whenever the inverse problem has a solution. We set
\begin{gather*}
\qquad\qquad\qquad A_*:=\sup_{x'\in\mathbb{R}^{n-1}}\Bigl(|a_{nn}(x')|+2\sum_{i<n}|a_{in}(x')|+|b_n(x')|\Bigr),\\
\qquad \lambda_w:=|\partial_{x_n}\varphi|_0+\frac{T^\alpha}{\Gamma(1+\alpha)} \|\partial_{x_n}f\|_0,\qquad
\lambda_v:=|\partial^2_{x_n}\varphi|_0+\frac{T^\alpha}{\Gamma(1+\alpha)}\|\partial^2_{x_n}f\|_0,
\end{gather*}
and, for $R>0$, the \emph{conditional (a priori) set}
\[
\qquad\qquad\qquad\qquad Q_R:=\bigl\{q\in C^{\ell,0}_b(\mathbb{R}^{n-1}\times[0,T]):\  \|q\|_0\le R\bigr\}.
\]
 Note that $Q_R\subset C^{\ell,0}_b(\mathbb{R}^{n-1}\times[0,T])$: the H\"{o}lder continuity of $q$ is required
for the direct problems to be classically solvable, but no bound on its H\"{o}lder seminorm
is imposed. Since every $q\in C^{\ell,0}_b(\mathbb{R}^{n-1}\times[0,T])$ belongs to $Q_R$ for some $R$,
uniqueness (Theorem~\ref{thm:B}) holds in the whole space $C^{\ell,0}_b(\mathbb{R}^{n-1}\times[0,T])$.

For two data sets $(\varphi,f,g)$ and $(\tilde\varphi,\tilde f,\tilde g)$, with $\mathcal G,\tilde{\mathcal G}$ given by \eqref{eq:calG}, the
\emph{data distance} is
\begin{multline}\label{eq:dist}
\varrho:=\sup_{\mathbb{R}^{n-1}\times(0,T]}|\mathcal G-\tilde{\mathcal G}|+\|g-\tilde g\|_0
+|\partial_{x_n}(\varphi-\tilde\varphi)|_0\\+|\nabla\partial_{x_n}(\varphi-\tilde\varphi)|_0+|\partial^2_{x_n}(\varphi-\tilde\varphi)|_0
+\|\partial_{x_n}(f-\tilde f)\|_0+\|\partial^2_{x_n}(f-\tilde f)\|_0 .
\end{multline}

\subsection*{6.3. Main results}

\begin{theorem}[Conditional Lipschitz stability]\label{thm:A}
Let both data sets satisfy Assumption~\ref{ass:S} (with the same operator $L$ and $g_0$), and let $q\in Q_R$, $\tilde q\in Q_R$ be solutions of the
corresponding inverse problems on $[0,T]$. Then
\begin{equation}\label{eq:stabA}
\qquad\qquad\qquad\qquad\qquad\qquad\qquad\qquad  \|q-\tilde q\|_0\le C_*\, \varrho,
\end{equation}
where $C_*$ depends only on $n,\alpha,T$, the constants in (A1), $g_0$, $R$, $\lambda_w$ and $\lambda_v$ (computed for the first data set).
An explicit value of $C_*$ is given in \eqref{eq:Cstar}.
\end{theorem}

\begin{theorem}[Global uniqueness]\label{thm:B}
Let Assumption~\ref{ass:S} hold. If $q_1,q_2\in C^{\ell,0}_b(\mathbb{R}^{n-1}\times[0,T])$ are two solutions
of the inverse problem \eqref{eq:main}--\eqref{eq:over} on $[0,T]$ with the same operator $L$
and the same data $(\varphi,f,g)$, then $q_1\equiv q_2$.
\end{theorem}

\begin{theorem}[Local existence in the general case]\label{thm:Dgen}
Let Assumption~\ref{ass:S} hold, let $\ell\le\gamma$ \textup{(}$\gamma$ the H\"{o}lder exponent of the coefficients\textup{)}, and
assume in addition\\
(D4) ${}^CD_t g$ extends from $\mathbb{R}^{n-1}\times(0,T]$ to a function in $ C^{\ell,0}_b(\mathbb{R}^{n-1}\times[0,T])$;\\
(D5) $\partial_{x_n}\varphi\in H^{2+\ell}(\mathbb{R}^n)$.
 
Then there is $T^*\in(0,T]$, depending only on the data and on the constants of Lemmas~\ref{lem:ops} and \ref{lem:gradK}
\textup{(}see \eqref{eq:TstarGen}\textup{)}, such that the inverse problem has a solution $q\in C^{\ell,0}_b(\mathbb{R}^{n-1}\times[0,T^*])$.
By Theorem~\ref{thm:B} this solution is unique.
\end{theorem}

Theorem~\ref{thm:Dgen} requires neither (A3) nor any smallness of the data. Condition (D5) is used only to control the
tangential derivatives $\partial_{x_i}w$, $i<n$, in \eqref{eq:traceII}; for $n=1$, and under (A3) (Theorem~\ref{thm:D} below),
these derivatives do not occur and (D5) is not needed.

\begin{theorem}[The case (A3): explicit estimate]\label{thm:C}
Under the hypotheses of Theorem~\ref{thm:A} and (A3), with $a^*:=\sup a_{nn}$, $V:=\lambda_v E_{\alpha}(RT^\alpha)$,
\begin{equation}\label{eq:stabC}
\qquad\qquad\qquad\qquad  \|q-\tilde q\|_0\le \varepsilon_0+\frac{a^*}{g_0}\Bigl[\varepsilon_v+\frac{VT^\alpha}{\Gamma(1+\alpha)}\varepsilon_0\Bigr]
E_{\alpha}\Bigl(\bigl(\tfrac{a^*V}{g_0}+R\bigr)T^\alpha\Bigr),
\end{equation}
where $\varepsilon_0=g_0^{-1}\bigl(\sup|\mathcal G-\tilde{\mathcal G}|+R \|g-\tilde g\|_0\bigr)$ and
$\varepsilon_v=|\partial^2_{x_n}(\varphi-\tilde\varphi)|_0+\frac{T^\alpha}{\Gamma(1+\alpha)}\|\partial^2_{x_n}(f-\tilde f)\|_0$.
\end{theorem}

\begin{theorem}[Local existence in the case (A3)]\label{thm:D}
Let Assumption~\ref{ass:S},  (A3) and (D4) hold, let $\ell\le\gamma$ \textup{(}$\gamma$ the H\"{o}lder exponent of the coefficients\textup{)}

Then there is $T^*\in(0,T]$, depending only on the data and on the constants of Lemma~\ref{lem:ops} \textup{(}see \eqref{eq:Tstar}\textup{)}, such
that the inverse problem has a solution $q\in C^{\ell,0}_b(\mathbb{R}^{n-1}\times[0,T^*])$. By Theorem~\ref{thm:B} this solution is unique.
\end{theorem}

\begin{remark}\label{rem:structure}
The proofs of Theorems~\ref{thm:A}, \ref{thm:B} and \ref{thm:C} use only three facts  from Sections 3--5: the integral representations \eqref{eq:IEw}, \eqref{eq:IEv}, the
sup-norm bounds of Lemma~\ref{lem:ops}(b), (c) (including $|\nabla_x K[G]|\le c_4\int_0^t(t-\tau)^{\alpha/2-1}\|G\|(\tau) d\tau$) and
the positivity/normalization Lemma~\ref{lem:norm}. All of them hold uniformly in the dimension (see Section 3); hence the stability and
uniqueness theorems are dimension-independent. The tangential derivatives $\partial_{x_i}w$ in \eqref{eq:traceII} cost only a
weaker time singularity $(t-\tau)^{\alpha/2-1}$, which is handled by a Gronwall inequality of order $\alpha/2$.
H\"{o}lder estimates are needed only for existence (Theorem~\ref{thm:Dgen}).
\end{remark}

\section*{7. Auxiliary results}\label{sec:aux}

\begin{lemma}[Gronwall inequality of order $\beta$]\label{lem:gr}
Let $0<\beta\le1$, let $\mu\ge0$ be bounded and measurable in  $(0,T]$, $A,B\ge0$, and
$\mu(t)\le A+B(I^\beta_{0+}\mu)(t)$ for $t\in(0,T]$. Then $\mu(t)\le A\,E_\beta(Bt^\beta)$.
\end{lemma}
\textit{Proof.} Identical to that of Lemma~\ref{lem:gronwall} with $\alpha$ replaced by $\beta$.

\begin{lemma}\label{lem:comp}
For  measurable bounded  $\mu\ge0$ in  $(0,T]$,
\begin{equation*}
\qquad\qquad\qquad\qquad (I^\alpha{0+}\mu)(t)\le\kappa\,(I^\beta_{0+}\mu)(t),\qquad \kappa:=\frac{\Gamma(\alpha/2)}{\Gamma(\alpha)}\,T^{\alpha/2}.    
\end{equation*}
\end{lemma}
\textit{Proof.} $(t-\tau)^{\alpha-1}=(t-\tau)^{\alpha/2}(t-\tau)^{\alpha/2-1}\le T^{\alpha/2}(t-\tau)^{\alpha/2-1}$.

\begin{lemma}[Measurability]\label{lem:meas}

(a) If $G$ is continuous on $\mathbb{R}^N\times J$, $J\subset[0,T]$ an interval, then $t\mapsto \|G\|(t)$ is lower semicontinuous on $J$, hence
Borel measurable.

(b) If $G\in C_b(\mathbb{R}^N\times[0,T])$ and $\psi$ is bounded and Lipschitz, then $\nabla_x K[G]$ and $\nabla_xZ_0\psi$ are continuous on
$\mathbb{R}^N\times(0,T]$. Consequently, for $z=Z_0\psi+ K[G]$ the function $t\mapsto \|\nabla_xz\|(t)$ is lower semicontinuous on $(0,T]$.
 
\end{lemma}
\textit{Proof.} 
(a) $\|G\|(t)$ is a supremum of the continuous functions $t\mapsto |G(x,t)|$.

(b) We use that $\nabla_xY(s,x,\xi)$ and $\nabla_xZ(s,x,\xi)$ are continuous in $(s,x)$ for $s>0$, $x\ne\xi$ (this is part of the
construction in \cite{EK}). Substituting $\xi=x+y$, $s=t-\tau$,
\[
\qquad\qquad\qquad \nabla_x K[G](x,t)=\int_0^T\!\!\int_{\mathbb{R}^N}\mathbf 1_{\{s<t\}}\,\nabla_xY(s,x,x+y)\,G(x+y,t-s) d y d s.
\]
where $\mathbf 1_{\{s<t\}}$ is the indicator function of the set $s<t.$

For $y\ne0$ and $s\ne t_0$ the integrand is continuous in $(x,t)$ at $(x_0,t_0)$, and by Lemma~\ref{lem:unified} it is dominated by
$C \|G\|_0\,s^{\alpha-1-(N+1)\alpha/2}\omega(|y| s^{-\alpha/2})$, whose integral over $(0,T)\times\mathbb{R}^N$ is
$C \|G\|_0\int_0^Ts^{\alpha/2-1} ds<\infty$. Dominated convergence gives continuity. For $\nabla_xZ_0\psi$ we use
$\nabla_x Z_0\psi(x,t)=\int\nabla_xZ(t,x,x+y)[\psi(x+y)-\psi(x)] dy$ (see the proof of Lemma~\ref{lem:ops}(c)) and the majorant
$C\Lambda\,t^{-(N+1)\alpha/2} |y|\,\omega(|y|t^{-\alpha/2})$, which is uniformly integrable for $t\ge t_0/2>0$.

\begin{lemma}[A priori bounds on the conditional set]\label{lem:apriori}
Let (A1)--(A2), (D1), (D2) hold and $q\in Q_R$. Then the regular solutions $w=w(q)$, $v=v(q)$ satisfy
\[
\qquad\qquad\qquad |w|\le W:=\lambda_w E_{\alpha}(RT^\alpha),\qquad |v|\le V:=\lambda_v E_{\alpha}(RT^\alpha)\qquad\text{on }\overline\Pi_T,
\]
and $\nabla_xw$ is bounded on $\Pi_T$.
\end{lemma}
\textit{Proof.} Theorem~\ref{thm:direct}(i),(iv) with $P=\|q\|_0\le R$ and the monotonicity of $E_{\alpha}$; $\partial_{x_n}\varphi$ is Lipschitz by  (D1).

\section*{8. Proofs of Theorems~\ref{thm:A} and ~\ref{thm:B}}\label{sec:stab}

\textit{Proof.} {\it Proof of Theorem~\ref{thm:A}}.  Let $(w,v)$, $(\tilde w,\tilde v)$ be the solutions of \eqref{eq:IEw}, \eqref{eq:IEv} corresponding to $(q,\varphi,f)$ and $(\tilde q,\tilde\varphi,\tilde f)$,
and put
\[
e(t)=\|q-\tilde q\|(t),\quad d_w=w-\tilde w,\quad d_v=v-\tilde v,\quad
\eta(t)=\|d_w\|(t)+\|\nabla_xd_w\|(t)+\|d_v\|(t)\ \ (0<t\le T).
\]
By Lemmas~\ref{lem:meas} and \ref{lem:apriori}, $e$ and $\eta$ are bounded and measurable on $(0,T]$.

\smallskip\noindent\emph{Step 1: the equation for $q$ controls $e$ by $\eta$ pointwise in time.}
Subtracting the identities \eqref{eq:traceII} for the two problems and writing $qg-\tilde q\tilde g=(q-\tilde q)g+\tilde q(g-\tilde g)$,
\[
\qquad\qquad\qquad\qquad (q-\tilde q)\,g=\bigl(S[q]-\tilde S[\tilde q]\bigr)-(\mathcal G-\tilde{\mathcal G})-\tilde q\,(g-\tilde g),
\]
where
$S[q]-\tilde S[\tilde q]=\bigl[a_{nn}d_v+2\sum_{i<n}a_{in}\partial_{x_i}d_w+b_nd_w\bigr]_{x_n=0}$. By \eqref{eq:g0} and $\|\tilde q\|_0\le R$, we have
\begin{equation}\label{eq:s1}
\qquad\qquad\qquad e(t)\le\varepsilon_0+\frac{A_*}{g_0}\,\eta(t),\qquad \varepsilon_0:=\frac{1}{g_0}\Bigl(\sup |\mathcal G-\tilde{\mathcal G}|+R\,\|g-\tilde g\|_0\Bigr).
\end{equation}

\smallskip\noindent\emph{Step 2: the direct problems \eqref{eq:Pw}, \eqref{eq:Pv} control
$\eta$ by integrals of $e$ and $\eta$.}
Subtracting the equations \eqref{eq:IEw} for the two problems and using
$qw-\tilde q\tilde w=(q-\tilde q)w+\tilde q\,d_w$, we obtain
\[
\qquad\qquad\qquad d_w=Z_0\,\partial_{x_n}(\varphi-\tilde\varphi)+K[\partial_{x_n}(f-\tilde f)]-K[(q-\tilde q)w]-K[\tilde q\,d_w];
\]
in the same way, \eqref{eq:IEv} and $qv-\tilde q\tilde v=(q-\tilde q)v+\tilde q\,d_v$ give
\[
\qquad\qquad\qquad d_v=Z_0\,\partial^2_{x_n}(\varphi-\tilde\varphi)+K[\partial^2_{x_n}(f-\tilde f)]-K[(q-\tilde q)v]-K[\tilde q\,d_v].
\]
Here $|(q-\tilde q)(\xi',\tau)w(\xi,\tau)|\le W e(\tau)$ and
$|(q-\tilde q)(\xi',\tau)v(\xi,\tau)|\le V e(\tau)$
(for $W$ and $V$ see Lemma~\ref{lem:apriori}).  By Lemma~\ref{lem:ops}(a), (b),
\begin{align}
\qquad\qquad\qquad\qquad\qquad \|d_w\|(t)&\le\varepsilon_w+W\,(I^{\alpha}_{0+} e)(t)+R\,(I^{\alpha}_{0+}\|d_w\|)(t),\label{eq:s2}\\
\qquad\qquad\qquad\qquad\qquad \|d_v\|(t)&\le\varepsilon_v+V\,(I^{\alpha}_{0+} e)(t)+R\,(I^{\alpha}_{0+}\|d_v\|)(t),\label{eq:s3}
\end{align}
with
$\varepsilon_w=|\partial_{x_n}(\varphi-\tilde\varphi)|_0+\frac{T^\alpha}{\Gamma(1+\alpha)}\|\partial_{x_n}(f-\tilde f)\|_0$ and $\varepsilon_v$ as in
Theorem~\ref{thm:C}. By Lemma~\ref{lem:ops}(c) (the function $\partial_{x_n}(\varphi-\tilde\varphi)$ is Lipschitz with constant
$|\nabla\partial_{x_n}(\varphi-\tilde\varphi)|_0$) and $\int_0^t(t-\tau)^{\alpha/2-1}\mu(\tau) d\tau=\Gamma(\alpha/2)(I^{\beta}_{0+}\mu)(t)$,
\begin{equation}\label{eq:s4}
\qquad\qquad\qquad \|\nabla_xd_w\|(t)\le\varepsilon_w'+c_4\Gamma(\tfrac\alpha2)\Bigl[W(I^{\beta}_{0+} e)(t)+R(I^{\beta}_{0+}\|d_w\|)(t)\Bigr],
\end{equation}
$\varepsilon_w'=c_3|\nabla\partial_{x_n}(\varphi-\tilde\varphi)|_0+\frac{2c_4}{\alpha}T^{\alpha/2}\|\partial_{x_n}(f-\tilde f)\|_0$.
Adding \eqref{eq:s2}--\eqref{eq:s4} and applying Lemma~\ref{lem:comp} to the $I^{\alpha}_{0+}$-terms,
\begin{equation}\label{eq:s5}
\qquad\qquad\qquad\qquad\qquad \eta(t)\le\varepsilon_1+k_1(I^{\beta}_{0+} e)(t)+k_2(I^{\beta}_{0+}\eta)(t),
\end{equation}
\[
\varepsilon_1=\varepsilon_w+\varepsilon_w'+\varepsilon_v,\qquad
k_1=\kappa(W+V)+c_4\Gamma(\tfrac\alpha2)W,\qquad
k_2=R\bigl(\kappa+c_4\Gamma(\tfrac\alpha2)\bigr).
\]

\smallskip\noindent\emph{Step 3: closing the system.}
Apply $I^{\beta}_{0+}$ to \eqref{eq:s1}: $(I^{\beta}_{0+} e)(t)\le\varepsilon_0t^{\alpha/2}/\Gamma(1+\alpha/2)+(A_*/g_0)(I^{\beta}_{0+}\eta)(t)$. Inserting this into \eqref{eq:s5},
\[
\qquad\qquad\qquad \eta(t)\le\Bigl[\varepsilon_1+\frac{k_1T^{\alpha/2}}{\Gamma(1+\alpha/2)}\varepsilon_0\Bigr]+B\,(I^{\beta}_{0+}\eta)(t),\qquad
B:=\frac{k_1A_*}{g_0}+k_2 .
\]
Lemma~\ref{lem:gr} with $\beta=\alpha/2$ gives
$\eta(t)\le\bigl[\varepsilon_1+k_1T^{\alpha/2}\varepsilon_0/\Gamma(1+\alpha/2)\bigr]E_{\beta}(BT^{\alpha/2})$, and \eqref{eq:s1} yields
\begin{equation}\label{eq:final}
\qquad\qquad\qquad \|q-\tilde q\|_0\le\varepsilon_0+\frac{A_*}{g_0}\Bigl[\varepsilon_1+\frac{k_1T^{\alpha/2}}{\Gamma(1+\alpha/2)}\varepsilon_0\Bigr]E_{\beta}\bigl(BT^{\alpha/2}\bigr).
\end{equation}
Since $\varepsilon_0\le g_0^{-1}\max\{1,R\}\varrho$ and $\varepsilon_1\le c\,\varrho$ with
$c=1+\frac{T^\alpha}{\Gamma(1+\alpha)}+c_3+\frac{2c_4}{\alpha}T^{\alpha/2}$, \eqref{eq:final} implies \eqref{eq:stabA} with
\begin{equation}\label{eq:Cstar}
\qquad\qquad\qquad C_*=\frac{\max\{1,R\}}{g_0}+\frac{A_*}{g_0}\Bigl[c+\frac{k_1T^{\alpha/2}\max\{1,R\}}{g_0\Gamma(1+\alpha/2)}\Bigr]E_{\beta}\bigl(BT^{\alpha/2}\bigr).
\end{equation}

\textit{Proof. Proof of Theorem~\ref{thm:B}.}
Let $q_1,q_2$ be solutions with the same operator $L$ and the same data $(\varphi,f,g)$. Both belong to $C^{\ell,0}_b$, hence are bounded; take
$R=\max\{\|q_1\|_0, \|q_2\|_0\}$. Then $q_1,q_2\in Q_R$, all terms of $\varrho$ vanish, and Theorem~\ref{thm:A} gives $q_1=q_2$.

\section*{9. Local existence in the general case: proof of Theorem~\ref{thm:Dgen}}\label{sec:exist}

Without (A3) the right-hand side of \eqref{eq:IEq} contains the tangential derivatives $\partial_{x_i}w$, $i<n$, on $\{x_n=0\}$,
and to keep it in $C^{\ell,0}_b$ we need H\"{o}lder bounds for the gradients of $K[G]$ and $Z_0\psi$. We first obtain  them for an operator
$\mathcal L$ satisfying (A1) in $\mathbb{R}^N$.

\begin{lemma}\label{lem:gradK}
Let $\ell\in(0,1)$. \\
(a) There is a constant $C$, depending on $N,\alpha,T$ and (A1), such that $\int_{\mathbb{R}^N}|D^2_xY(s,x,\xi)| d\xi\le Cs^{-1}$ for
$0<s\le T$, $x\in\mathbb{R}^N$.\\
(b) There is a constant $c_5$, depending on $N,\alpha,\ell,T$ and (A1), such that for bounded continuous $G$ and $t\in(0,T]$
\[
\qquad\qquad [\nabla_x K[G](\cdot,t)]_{\ell}\le c_5\int_0^t(t-\tau)^{\alpha(1-\ell)/2-1} \|G\|(\tau)d\tau
\le\frac{2c_5\,t^{\alpha(1-\ell)/2}}{\alpha(1-\ell)}\sup_{\tau\le t} \|G\|(\tau).
\] 
\end{lemma}

\textit{Proof.}
(a) By \eqref{eq:Y0} and the homogeneity used in Section 2,
\[
\qquad\qquad D^2_xY_0(s,x-\xi;\xi)=s^{\alpha-1-(N+2)\alpha/2}(\det A(\xi))^{-1/2}(D^2\Psi_{N,\xi})\bigl(s^{-\alpha/2}(x-\xi)\bigr)
\]
with
$\Psi_{N,\xi}(z)=\Psi_N\bigl(|A(\xi)^{-1/2}z|\bigr)$. By Table~\ref{tab:sing} (column $|D^2\Psi_N|$), the exponential decay
\cite[(2.11)]{EK} and the uniformity of the constants in $\xi$, $|D^2\Psi_{N,\xi}(z)|\le\Omega(|z|)$ with a function $\Omega$
independent of $\xi$ and integrable over $\mathbb{R}^N$. The substitution $\xi=x-s^{\alpha/2}z$ gives
\[
\qquad\qquad\qquad \int_{\mathbb{R}^N}|D^2_xY_0(s,x-\xi;\xi)| d\xi\le Cs^{\alpha-1-\alpha}\int_{\mathbb{R}^N}\Omega(|z|) d z=Cs^{-1}.
\]
For the correction, \cite[Theorem (a)(vi)]{EK}
($N\ge2$) gives $|D^2_xV_Y|\le Cs^{\gamma_1-1}|x-\xi|^{-N+\gamma_2}e^{-\sigma\rho}$, whose integral in $\xi$ is
$Cs^{\gamma_1-1+\alpha\gamma_2/2}\le Cs^{-1}$; for $N=1$, \cite[Theorem (a)(i)]{EK} with $m=2$ gives in the same way the bound
$Cs^{\gamma\alpha/2-1}\le Cs^{-1}$.

(b) By Lemma~\ref{lem:ops}(c), $\nabla_x K[G]$ may be computed under the integral sign. Let $h=|x-\tilde x|>0$ and $s>0$. For almost
every $\xi$ the segment $x_\theta=x+\theta(\tilde x-x)$, $\theta\in[0,1]$, avoids $\xi$, so
$\nabla_xY(s,x,\xi)-\nabla_xY(s,\tilde x,\xi)=-\int_0^1D^2_xY(s,x_\theta,\xi)(\tilde x-x) d\theta$. By Fubini's theorem and (a),
$\int|\nabla_xY(s,x,\xi)-\nabla_xY(s,\tilde x,\xi)| d\xi\le Chs^{-1}$, while Lemma~\ref{lem:int} \textup{(}$  |m|=1$, $\beta=0$\textup{)}
bounds the same integral by $2Cs^{\alpha/2-1}$. Interpolating, $\min\{a,b\}\le a^{1-\ell}b^{\ell}$, we get
\[
\qquad\qquad \int_{\mathbb{R}^N}|\nabla_xY(s,x,\xi)-\nabla_xY(s,\tilde x,\xi)| d\xi\le c_5\,h^{\ell}s^{(\alpha/2-1)(1-\ell)-\ell}=c_5\,h^{\ell}s^{\alpha(1-\ell)/2-1}.
\]
Integrating against $\|G\|(\tau)$ with $s=t-\tau$ gives the first inequality; the second follows from
$\int_0^t(t-\tau)^{\alpha(1-\ell)/2-1} d\tau=2t^{\alpha(1-\ell)/2}/(\alpha(1-\ell))$.

\begin{lemma}\label{lem:Zsmooth}
Let $\ell\le\gamma$ and $\psi\in H^{2+\ell}(\mathbb{R}^N)$. Then
\begin{equation}\label{eq:Zsmooth}
\qquad\qquad\qquad\qquad Z_0\psi=\psi+ K[\mathcal L\psi]\qquad\text{on }\qquad \mathbb{R}^N\times[0,T].
\end{equation}
Consequently $\nabla_xZ_0\psi$ is bounded and continuous on $\mathbb{R}^N\times[0,T]$, $\nabla_x Z_0\psi(\cdot,0)=\nabla\psi$, and for $t\in[0,T],$ we have
\[
|\nabla_xZ_0\psi(\cdot,t)|_0\le|\nabla\psi|_0+\frac{2c_4}{\alpha}t^{\alpha/2}|\mathcal L\psi|_0,\qquad
[\nabla_xZ_0\psi(\cdot,t)]_{\ell}\le [\nabla\psi]_{\ell}+\frac{2c_5\,t^{\alpha(1-\ell)/2}}{\alpha(1-\ell)}|\mathcal L\psi|_0 .
\]
\end{lemma}
\textit{Proof.}
Since the coefficients of $\mathcal L$ are bounded and H\"{o}lder continuous with exponent $\gamma\ge\ell$, they belong to $H^\ell(\mathbb{R}^N)$,
hence $\mathcal L\psi\in H^\ell(\mathbb{R}^N)$. By Theorem~\ref{thm:EK}(ii) with zero initial value, $ K[\mathcal L\psi]$ is a bounded classical
solution of ${}^CD_t z-\mathcal Lz=\mathcal L\psi$, $z(\cdot,0)=0$, and $ K[\mathcal L\psi](x,\cdot)\in H^{\alpha+\lambda}_\nu[0,T]$. The function
$\psi$ satisfies ${}^CD_t\psi=0$, and $\psi(x)\in H^{\alpha+\lambda}_\nu[0,T]$ as in the proof of
Lemma~\ref{lem:norm}. Hence $z=\psi+ K[\mathcal L\psi]$ is a bounded classical solution of ${}^CD_t z-\mathcal Lz=0$, $z(\cdot,0)=\psi$, in the
uniqueness class of Theorem~\ref{thm:EK}(iii). The same argument applies to $Z_0\psi$ by Theorem~\ref{thm:EK}(ii), which yields \eqref{eq:Zsmooth}.

By \eqref{eq:Zsmooth}, $\nabla_x Z_0\psi=\nabla\psi+\nabla_x K[\mathcal L\psi]$. The function $\nabla_x K[\mathcal L\psi]$ is continuous on
$\mathbb{R}^N\times(0,T]$ by Lemma~\ref{lem:meas}(b), and $|\nabla_x K[\mathcal L\psi](x,t)|\le\frac{2c_4}{\alpha}t^{\alpha/2} |\mathcal L\psi|_0$ by
Lemma~\ref{lem:ops}(c), so it extends continuously by $0$ to $t=0$. The H\"{o}lder bound follows from Lemma~\ref{lem:gradK}(b).

\textit{Proof.}[Proof of Theorem~\ref{thm:Dgen}]
For $n=1$ there are no tangential derivatives, and the proof below applies with all terms containing $\partial_{x_i}$, $i<n$, omitted.
We use the notation of \eqref{eq:IEq}. For $0<T^*\le T,$ we put $X^*=\mathbb{R}^{n-1}\times[0,T^*]$ and
\[
\qquad\qquad\qquad \omega_1(\tau)=\frac{\tau^\alpha}{\Gamma(1+\alpha)}+\frac{c_2\,\tau^{\alpha(1-\ell/2)}}{\alpha(1-\ell/2)},\qquad
\omega_2(\tau)=\frac{2c_4}{\alpha}\tau^{\alpha/2}+\frac{2c_5\,\tau^{\alpha(1-\ell)/2}}{\alpha(1-\ell)} .
\]
By Lemma~\ref{lem:ops}(b),(c) and Lemma~\ref{lem:gradK}(b), for bounded continuous $G$ on $\mathbb{R}^n\times[0,T^*]$,
\begin{equation}\label{eq:KGl}
\qquad\qquad\qquad\qquad \|K[G]\|_\ell\le\omega_1(T^*) \|G\|_0,\qquad \|\nabla_x K[G]\|_\ell\le\omega_2(T^*)\|G\|_0
\end{equation}
\textup{(}norms over $\mathbb{R}^n\times[0,T^*]$\textup{)}; moreover, $ K[G]$ and $\nabla_x K[G]$ are continuous on $\mathbb{R}^n\times[0,T^*]$ and vanish
at $t=0$ (Lemma~\ref{lem:meas}(b) and the sup-norm bounds of Lemma~\ref{lem:ops}(b),(c)).

\emph{$q_0\in C^{\ell,0}_{b}(\mathbb{R}^{n-1}\times[0,T])$.} By (D1), (D2) and Lemma~\ref{lem:ops}(a),(b), $v_0,w_0\in C^{\ell,0}_{b}(\overline\Pi_T)$. By (D5)
and Lemma~\ref{lem:Zsmooth} with $\mathcal L=L$, $\psi=\partial_{x_n}\varphi$, and by \eqref{eq:KGl} with $G=\partial_{x_n}f$, the derivatives
$\partial_{x_i}w_0=\partial_{x_i}Z_0\,\partial_{x_n}\varphi+\partial_{x_i} K[\partial_{x_n}f]$ belong to $C^{\ell,0}_{b}(\overline\Pi_T)$. Restrictions to
$\{x_n=0\}$ do not increase these norms. By (D2), (D3) and (D4), $f(x',0,t)$, $L'g$ and ${}^CD_t g$ belong to $C^{\ell,0}_{b}$. Since
$\ell\le\gamma$, the coefficients belong to $H^\ell(\mathbb{R}^{n-1})$; $C^{\ell,0}_{b}$ is an algebra with $\|ab\|_\ell\le\|a\|_\ell \|b\|_\ell$, and
$1/g\in C^{\ell,0}_{b}$ because $[1/g]_{\ell}\le [g]_{\ell}/g_0^2$. Hence $q_0\in C^{\ell,0}_{b}(\mathbb{R}^{n-1}\times[0,T])$.

\emph{The map.} For $q\in C^{\ell,0}_{b}(X^*)$ let $w=w(q)$, $v=v(q)$ be the regular solutions of \eqref{eq:Pw}, \eqref{eq:Pv} on $[0,T^*]$, i.e.\ the
solutions of \eqref{eq:IEw}, \eqref{eq:IEv}, and put
\[
\qquad\qquad\qquad \mathcal P(q):=q_0-\frac1g\,\mathcal B[q]=\frac1g\Bigl(S[q]-\mathcal G\Bigr)\qquad\text{on } \qquad X^*,
\]
with $S[q]$ and $\mathcal G$ as in \eqref{eq:traceII}, \eqref{eq:calG}. By Proposition~\ref{prop:system} (applied on $[0,T^*]$) and the
compatibility condition, a fixed point of $\mathcal P$ in $ C^{\ell,0}_{b}(X^*)$ is a solution of the inverse problem on $[0,T^*]$.

\emph{Invariance of a ball.} Fix $\rho>0$, let $R=\|q_0\|_0+\rho$, let $W,V$ be the constants of Lemma~\ref{lem:apriori} for this $R$, and
\[
\qquad\qquad\qquad\qquad\qquad \mathcal B_\rho=\{q\in C^{\ell,0}_{b}(X^*):\ \|q-q_0\|_\ell\le\rho\}.
\]
For $q\in\mathcal B_\rho$ we have $\|q_0\|\le R$, hence $|w|\le W$, $|v|\le V$ (Lemma~\ref{lem:apriori} on $[0,T^*]$), and
$\|qw\|_0\le RW$, $\|qv\|_0\le RV$. With
$m_\ell:=\|1/g\|_\ell\bigl(|a_{nn}|_\ell+2\sum_{i<n}|a_{in}|_\ell+|b_n|_\ell\bigr)$ (H\"{o}lder norms on $\mathbb{R}^{n-1}$), the
product rule in $ C^{\ell,0}_{b}$ and \eqref{eq:KGl} give
\[
\Bigl\|\frac1g\,\mathcal B[q]\Bigr\|_\ell\le m_\ell\max\bigl\{\| K[qv]\|_\ell,\ \|\nabla_x K[qw]\|_\ell,\ \| K[qw]\|_\ell\bigr\}
\le m_\ell R(V+W)\bigl(\omega_1(T^*)+\omega_2(T^*)\bigr),
\]
and $\mathcal B[q]$ is continuous on $X^*$. With $k_1,k_2$ the constants of \eqref{eq:s5} \textup{(}computed with the above $R$, $W$,
$V$\textup{)}, choose $T^*\in(0,T]$ such that
\begin{equation}\label{eq:TstarGen}
m_\ell R(V+W)\bigl(\omega_1(T^*)+\omega_2(T^*)\bigr)\le\rho,\qquad
\theta_*:=\frac{A_*}{g_0}\,\frac{k_1(T^*)^{\alpha/2}}{\Gamma(1+\alpha/2)}\,E_{\beta}\bigl(k_2T^{\alpha/2}\bigr)<1 .
\end{equation}
Then $\mathcal P(q)\in C^{\ell,0}_{b}(X^*)$ and $\|\mathcal P(q)-q_0\|_\ell\le\rho$, i.e.\ $\mathcal P(\mathcal B_\rho)\subset\mathcal B_\rho$.

\emph{Contraction in the sup-norm.} Let $q,\tilde q\in\mathcal B_\rho$, $\delta=\|q-\tilde q\|_0$, $d_w=w(q)-w(\tilde q)$,
$d_v=v(q)-v(\tilde q)$, $e(t)=\|q-\tilde q\|(t)\le\delta$ and $\eta(t)=\|d_w\|(t)+\|\nabla_xd_w\|(t)+\|d_v\|(t)$. Since the data of
\eqref{eq:Pw}, \eqref{eq:Pv} are the same for $q$ and $\tilde q$, $d_w=- K[(q-\tilde q)w(q)]- K[\tilde q\,d_w]$, and similarly for $d_v$.
Step~2 of Section 7 with $\varepsilon_w=\varepsilon_w'=\varepsilon_v=0$ gives
$\eta(t)\le k_1(I^{\beta}_{0+} e)(t)+k_2(I^{\beta}_{0+}\eta)(t)\le\frac{k_1\delta(T^*)^{\alpha/2}}{\Gamma(1+\alpha/2)}+k_2(I^{\beta}_{0+}\eta)(t)$ on $(0,T^*]$; $\eta$ is bounded and
measurable by Lemmas~\ref{lem:meas} and \ref{lem:apriori}. Lemma~\ref{lem:gr} with $\beta=\alpha/2$ yields
$\eta(t)\le\frac{k_1\delta(T^*)^{\alpha/2}}{\Gamma(1+\alpha/2)}E_{\beta}\bigl(k_2T^{\alpha/2}\bigr)$. Since
$\mathcal P(q)-\mathcal P(\tilde q)=\frac1g\bigl[a_{nn}d_v+2\sum_{i<n}a_{in}\partial_{x_i}d_w+b_nd_w\bigr]_{x_n=0}$, we obtain
$\|\mathcal P(q)-\mathcal P(\tilde q)\|_0\le\frac{A_*}{g_0}\sup_{t\le T^*}\eta(t)\le\theta_*\|q-\tilde q\|_0$.

\emph{Conclusion.} The set $\mathcal B_\rho$ with the metric $\|q-\tilde q\|_0$ is complete:
a Cauchy sequence $(q_k)\subset\mathcal B_\rho$ in the sup-norm converges uniformly to a
bounded continuous function $q$, and since both $\|\cdot\|_0$ and
$\sup_t[(\cdot)(\cdot,t)]_{\ell}$ are lower semicontinuous under pointwise convergence,
\[
\qquad \|q-q_0\|_\ell\le\liminf_{k\to\infty}\|q_k-q_0\|_0+\liminf_{k\to\infty}\sup_t[(q_k-q_0)(\cdot,t)]_{\ell}
\le\liminf_{k\to\infty}\|q_k-q_0\|_\ell\le\rho,
\]
i.e.\ $q\in\mathcal B_\rho$. By the contraction mapping principle  $\mathcal P$ has a unique fixed point $q\in\mathcal B_\rho$, which
solves the inverse problem on $[0,T^*]$; its uniqueness in the whole of $ C^{\ell,0}_{b}(X^*)$ follows from
Theorem~\ref{thm:B} with $T$ replaced by $T^*$.

\emph{Invariance of a ball.} Fix $\rho>0$, let $R=\|q_0\|_0+\rho$ and
\[
\qquad\qquad\qquad\qquad\qquad \mathcal B=\{q\in C^{\ell,0}_{b}(X^*):\ \|q-q_0\|_\ell\le\rho\}.
\]
For $q\in\mathcal B$ we have $\|q\|_0\le R$, so $|v(q)|\le V=\lambda_v E_{\alpha}(RT^\alpha)$ (Lemma~\ref{lem:apriori}). By
Lemma~\ref{lem:ops}(b), which uses only the sup-norm of the density $qv(q)$,
\[
\qquad\qquad\qquad\qquad \| K[qv(q)](\cdot,0,\cdot)\|_\ell\le RV\Bigl(\frac{(T^*)^\alpha}{\Gamma(1+\alpha)}+\frac{c_2(T^*)^{\alpha(1-\ell/2)}}{\alpha(1-\ell/2)}\Bigr).
\]
With $m_\ell:=\|a_{nn}/g\|_\ell$, $m_0:=\sup|a_{nn}/g|$ choose $T^*\in(0,T]$ such that
\begin{equation}\label{eq:Tstar}
m_\ell RV\Bigl(\frac{(T^*)^\alpha}{\Gamma(1+\alpha)}+\frac{c_2(T^*)^{\alpha(1-\ell/2)}}{\alpha(1-\ell/2)}\Bigr)\le\rho,\qquad
\theta:=m_0VE_{\alpha}(RT^\alpha)\frac{(T^*)^\alpha}{\Gamma(1+\alpha)}<1 .
\end{equation}
Then $\|\mathcal T(q)-q_0\|_\ell\le\rho$, i.e.\ $\mathcal T(\mathcal B)\subset\mathcal B$. (Continuity of $\mathcal T(q)$ up to $t=0$ follows from
that of $v(q)$, Theorem~\ref{thm:direct}.)

\section*{10. The case (A3)} 

\subsection*{10.1. Proof of Theorem~\ref{thm:C}}
Under (A3) the function $w$ does not enter \eqref{eq:traceII}, which reads $qg=a_{nn}v|_{x_n=0}-\mathcal G$ (Corollary~\ref{cor:A3}). Step 1
of Section 7 gives $e(t)\le\varepsilon_0+(a^*/g_0)\|d_v\|(t)$, and \eqref{eq:s3} becomes
\[
\qquad\qquad\qquad \|d_v\|(t)\le\varepsilon_v+\frac{VT^\alpha}{\Gamma(1+\alpha)}\varepsilon_0+\Bigl(\frac{a^*V}{g_0}+R\Bigr)(I^{\alpha}_{0+}\|d_v\|)(t).
\]
Applying Lemma~\ref{lem:gr} with $\beta=\alpha$ and using the bound for $e,$ we obtain  \eqref{eq:stabC}. \qed

\begin{remark}
For $L=\Delta$ ($a^*=1$) and $n=1,2,3$, \eqref{eq:stabC} is the stability estimate of \cite[Theorem 2]{DD21}, \cite{DRB21}, \cite{UD21},
with the constant $C_0\Gamma(\alpha)$ appearing there equal to $1$ (see Remark~\ref{rem:compare}) and with the Mittag-Leffler function in place of
the series obtained by successive approximation. Here the order of the Gronwall inequality is $\alpha$; without (A3) it is $\alpha/2$.
\end{remark}

\subsection*{10.2. Proof of Theorem~\ref{thm:D}}
Put
\[
q_0:=\frac{1}{g}\Bigl(a_{nn}\,v_0(x',0,t)-\mathcal G\Bigr)
=\frac1g\Bigl\{f(x',0,t)-{}^CD_t g+L'g+a_{nn}v_0(x',0,t)\Bigr\},
\qquad v_0=Z_0\,\partial^2_{x_n}\varphi+ K[\partial^2_{x_n}f].
\]
\emph{$q_0\in C^{\ell,0}_{b}(\mathbb{R}^{n-1}\times[0,T])$.} Since $\ell\le\gamma$ and the coefficients are bounded, they belong to $H^\ell$. By (D3),
(D3),  (D4) the functions $f(x',0,t)$, ${}^CD_t g$, $L'g$ belong to $ C^{\ell,0}_{b}$; $v_0\in C^{\ell,0}_{b}(\overline\Pi_T)$ by Lemma~\ref{lem:ops}(a),(b), and its
restriction to $\{x_n=0\}$ has no larger norm. $ C^{\ell,0}_{b}$ is an algebra, $\|ab\|_\ell\le \|a\|_\ell\|b\|_\ell$, and $1/g\in C^{\ell,0}_{b}$ because
$[1/g]_{\ell}\le [g]_{\ell}/g_0^2$.

\emph{The map.} For $0<T^*\le T$ let $X^*=\mathbb{R}^{n-1}\times[0,T^*]$. For $q\in C^{\ell,0}_{b}(X^*)$ let $v(q)$ be the regular solution of
\eqref{eq:Pv} on $[0,T^*]$ and
\[
\qquad\qquad\qquad \mathcal T(q):=\frac1g\Bigl(a_{nn}\,v(q)(x',0,t)-\mathcal G\Bigr)=q_0-\frac{a_{nn}}{g}\, K[q\,v(q)](x',0,t),
\]
where we used $v(q)=v_0- K[qv(q)]$ (see \eqref{eq:IEv}). By Corollary~\ref{cor:A3} and the compatibility condition, a fixed point of $\mathcal T$ in
$ C^{\ell,0}_{b}(X^*)$ is a solution of the inverse problem on $[0,T^*]$.

\emph{Contraction in the sup-norm.} For $q,\tilde q\in\mathcal B$, $\mathcal T(q)-\mathcal T(\tilde q)=\frac{a_{nn}}{g}\bigl(v(q)-v(\tilde q)\bigr)|_{x_n=0}$.
By Theorem~\ref{thm:stab}, \eqref{eq:stab0}, with equal initial data and right-hand sides,
$\|v(q)-v(\tilde q)\|_0\le\frac{(T^*)^\alpha}{\Gamma(1+\alpha)}V\|q-\tilde q\|_0 E_{\alpha}(RT^\alpha)$, so
$\|\mathcal T(q)-\mathcal T(\tilde q)\|_0\le\theta\|q-\tilde q\|_0$.

Since $\mathcal B$ is complete with respect to the metric $\|q-\tilde q\|_0$ (see the proof of
Theorem~\ref{thm:Dgen}), the contraction mapping principle gives a unique fixed point of
$\mathcal T$ in $\mathcal B$; uniqueness in the whole of $ C^{\ell,0}_{b}(X^*)$ follows from
Theorem~\ref{thm:B} with $T$ replaced by $T^*$. \qed

\section*{11. An example} 

We construct an explicit example of the inverse problem in dimension $n=2$, $x=(x_1,x_2)$, $x'=x_1$, for an operator that satisfies
(A1), (A2) but not (A3). It shows that all hypotheses of Theorems~\ref{thm:A}, \ref{thm:B} and \ref{thm:Dgen} can be
met, and that the terms with the mixed coefficient and the normal drift in the trace condition \eqref{eq:traceII} cannot be omitted.

\smallskip\noindent\emph{The operator.} Let $b\in\mathbb{R}$ and
\[
\qquad\qquad\qquad\qquad L=\frac{\partial^2}{\partial x_1^2}+\cos x_1\,\frac{\partial^2}{\partial x_1\partial x_2}+\frac{\partial^2}{\partial x_2^2}+b\,\frac{\partial}{\partial x_2},
\]
i.e.\ $a_{11}=a_{22}=1$, $a_{12}=a_{21}=\frac12\cos x_1$, $b_1=0$, $b_2=b$. The eigenvalues of $(a_{ij}(x_1))$ are $1\pm\frac12 |\cos x_1|\ge\frac12$,
so (A1) holds with $\delta=\frac12$ and any $\gamma\in(0,1)$; the coefficients depend only on $x_1$, so (A2) holds; (A3) fails because
$a_{12}\not\equiv0$ (and $b_2\ne0$ if $b\ne0$). Here $L'=\partial^2_{x_1}$.

\smallskip\noindent\emph{The data.} Fix $\ell\in(0,1)$ and an arbitrary $q^*\in C^{\ell,0}_{b}(\mathbb{R}\times[0,T])$, for instance
$q^*(x_1,t)=(1+t)\bigl(1+\frac12\sin x_1\bigr)$. Put
\[
m(t)=1+t^\alpha,\qquad X(x_1)=2+\sin x_1,\qquad Y(x_2)=2+\sin x_2,\qquad u^*(x,t)=m(t)X(x_1)Y(x_2).
\]
By \eqref{eq:caputo}, $I^{1-\alpha}_{0+}[t^\alpha]=\Gamma(1+\alpha)\,t$, hence ${}^CD_t m=\Gamma(1+\alpha)$. A direct computation gives
\[
Lu^*=m(t)\Bigl[-\sin x_1\,Y(x_2)+\cos^2x_1\cos x_2-X(x_1)\sin x_2+b\,X(x_1)\cos x_2\Bigr].
\]
We define the data of the inverse problem by
\begin{align*}
\qquad\qquad\qquad \varphi(x)&=u^*(x,0)=X(x_1)Y(x_2),\\
f(x,t)&={}^CD_t u^*-Lu^*+q^*u^*=\Gamma(1+\alpha)X(x_1)Y(x_2)+q^*(x_1,t)\,m(t)X(x_1)Y(x_2)\\
&\quad+m(t)\Bigl[\sin x_1\,Y(x_2)-\cos^2x_1\cos x_2+X(x_1)\bigl(\sin x_2-b\cos x_2\bigr)\Bigr],\\
g(x_1,t)&=u^*(x_1,0,t)=2\,m(t)X(x_1).
\end{align*}

\smallskip\noindent\emph{Verification of the hypotheses.} The functions $\varphi$, $X$, $Y$ are smooth with bounded derivatives of all orders, so
$\varphi\in H^{2+\ell}(\mathbb{R}^2)$ and $\partial_{x_2}\varphi\in H^{2+\ell}(\mathbb{R}^2)$, which is (D1) and (D5). The function $f$ and its
derivatives $\partial_{x_2}f$, $\partial^2_{x_2}f$ are bounded, continuous on $\mathbb{R}^2\times[0,T]$ and Lipschitz in $x$ uniformly in $t$, except for the
factor $q^*$, which belongs to $ C^{\ell,0}_{b}$; hence (D2) holds. Next, $g$ and $\partial_{x_1}g$, $\partial^2_{x_1}g$ belong to $ C^{\ell,0}_{b}(\mathbb{R}\times[0,T])$;
$I^{1-\alpha}_{0+}[g(x_1,\cdot)-g(x_1,0)]=2\Gamma(1+\alpha)X(x_1)\,t$ is $C^1$; and $t^\nu g(x_1,t)=2X(x_1)(t^\nu+t^{\nu+\alpha})$ is H\"{o}lder
continuous with exponent $\alpha+\lambda$, since $\nu>\alpha+\lambda$. Thus (D3) holds, and
\[
\qquad\qquad\qquad\qquad\qquad\qquad {}^CD_t g(x_1,t)=2\Gamma(1+\alpha)X(x_1)
\]
extends continuously to $t=0$, which is (D4). Finally, $\varphi(x_1,0)=2X(x_1)=g(x_1,0)$ and $|g|\ge2$, so \eqref{eq:g0} holds
with $g_0=2$. Hence Assumption~\ref{ass:S}, (D4) and (D5) are satisfied.

\smallskip\noindent\emph{The solution.} The function $u^*$ is bounded, H\"{o}lder continuous in $x$ uniformly in $t$ and smooth in $x$; moreover
$I^{1-\alpha}_{0+}[u^*(x,\cdot)-u^*(x,0)]=\Gamma(1+\alpha)X(x_1)Y(x_2)\,t$ and $t^\nu u^*(x,t)$ is H\"{o}lder continuous with exponent $\alpha+\lambda$.
Hence $u^*\in \mathcal{U}_T(\mathbb{R}^2)$, and by construction it satisfies \eqref{eq:main}, \eqref{eq:init} with $q=q^*$. By Theorem~\ref{thm:direct}(iii),
$u^*$ is the regular solution of the direct problem with coefficient $q^*$, and $u^*(x_1,0,t)=g(x_1,t)$. Therefore $q^*$ is a solution of
the inverse problem on $[0,T]$, and by Theorem~\ref{thm:B} it is the only solution in $ C^{\ell,0}_{b}(\mathbb{R}\times[0,T])$.

\smallskip\noindent\emph{The trace condition.} By Theorem~\ref{thm:diff},
\[
\qquad\qquad\qquad w=\partial_{x_2}u^*=m(t)X(x_1)\cos x_2,\qquad v=\partial^2_{x_2}u^*=-m(t)X(x_1)\sin x_2 .
\]
On $\{x_2=0\}$,
\[
\qquad\qquad\qquad\qquad v=0,\qquad w=m(t)X(x_1),\qquad \partial_{x_1}w=m(t)\cos x_1,
\]
so that
\[
\qquad\qquad S[q^*]=a_{22}v+2a_{12}\,\partial_{x_1}w+b_2w\Big|_{x_2=0}=m(t)\bigl(\cos^2x_1+b\,X(x_1)\bigr).
\]
On the other hand, by \eqref{eq:calG}, with $L'g=-2m(t)\sin x_1$,
\[
\qquad \mathcal G={}^CD_t g-L'g-f(x_1,0,t)=m(t)\bigl(\cos^2x_1+b\,X(x_1)\bigr)-2q^*(x_1,t)\,m(t)X(x_1).
\]
Consequently,
\[
\qquad\qquad\qquad\qquad\qquad \frac{S[q^*]-\mathcal G}{g}=\frac{2q^*m(t)X(x_1)}{2m(t)X(x_1)}=q^*(x_1,t),
\]
i.e.\ \eqref{eq:traceII} holds and $q^*$ is recovered exactly from the data $(\varphi,f,g)$.

\smallskip\noindent\emph{The role of the mixed coefficient and the normal drift.} In this example $v=0$ on $\{x_2=0\}$, so the whole
contribution of the normal derivatives to the trace condition comes from $\partial_{x_1}w$ and $w$, i.e.\ from $a_{12}$ and $b_2$. If these terms
were dropped, as is legitimate only under (A3) (cf.\ \eqref{eq:IEqA3}), the formula would give
\[
\qquad\qquad\qquad\hat q=\frac{a_{22}v|_{x_2=0}-\mathcal G}{g}=q^*-\frac{\cos^2x_1+b\,(2+\sin x_1)}{2\,(2+\sin x_1)}\ne q^*.
\]
Thus, without (A3), the tangential derivative $\partial_{x_1}w$ of the normal derivative and the trace of $w$ are genuinely needed, and the
equivalent system must contain the equation \eqref{eq:IEw} for $w$; this is the situation covered by Theorems~\ref{thm:A}, \ref{thm:B}
and \ref{thm:Dgen}, but not by Theorem~\ref{thm:D}.

\smallskip\noindent\emph{Stability.} Let $h\in C^{\ell,0}_{b}(\mathbb{R}\times[0,T])$, $\varepsilon>0$, and let $(\tilde\varphi,\tilde f,\tilde g)$ be the data constructed
as above with $q^*$ replaced by $\tilde q^*=q^*+\varepsilon h$ and the same $u^*$. Then $\tilde\varphi=\varphi$, $\tilde g=g$ and
$\tilde f-f=\varepsilon h\,u^*$; since $\mathcal G-\tilde{\mathcal G}=(\tilde f-f)(x_1,0,t)=\varepsilon h\,m(t)X(x_1)Y(0)$, in \eqref{eq:dist} we have
$$
\varrho=\sup |\mathcal G-\tilde{\mathcal G}|+\|\partial_{x_2}(f-\tilde f)\|_0+\|\partial^2_{x_2}(f-\tilde f)\|_0\qquad\qquad\qquad\qquad$$$$
\qquad\qquad\qquad\le\varepsilon \|h\|_0\sup(mX)\bigl(2+\sup |Y'|+\sup |Y''|\bigr)\le12(1+T^\alpha)\,\varepsilon \|h\|_0,
$$
while $\|q^*-\tilde q^*\|_0=\varepsilon \|h\|_0$. Both quantities are of order $\varepsilon$, in agreement with the Lipschitz estimate
\eqref{eq:stabA} of Theorem~\ref{thm:A}.

\begin{remark}
An analogous construction works in any dimension $n\ge2$: one takes
$u^*=m(t)\prod_{i=1}^{n}(2+\sin x_i)$, an operator $L$ with
coefficients depending on $x'$ only, an arbitrary $q^*\in C^{\ell,0}_{b}(\mathbb{R}^{n-1}\times[0,T])$, and defines $\varphi$, $f$, $g$ as above. For $n=1$,
$u^*=m(t)(2+\sin x)$ and $L=a\,\partial_x^2+b\,\partial_x$ give an example for the problem of \cite{DD21} with a drift.
\end{remark}

\section*{12. Discussion} 

\begin{enumerate} 
\item \emph{Where the dimension enters.} The pointwise behaviour of the kernels differs for $n=1,2,3,4$ and $n\ge5$
(Tables~\ref{tab:sing}, \ref{tab:V}): $Y$ is bounded near the diagonal only for $n\le3$, has a logarithmic singularity for $n=4$ and
a power singularity $|x-\xi|^{4-n}$ for $n\ge5$, and the corrections $V_Z,V_Y$ obey different estimates, with a different gain for $n=4$.
After normalization by $r=|x-\xi|t^{-\alpha/2}$  are dominated by the single profile $\omega$ in \eqref{eq:omega}, and every
estimate used for the direct problem is an \emph{integral} bound with the same power of $t$ in all dimensions (Lemma~\ref{lem:int}).
The proofs of stability and uniqueness for the inverse problem use only these integral bounds together with Lemmas~\ref{lem:norm} and
\ref{lem:ops} (Remark~\ref{rem:structure}). This is why Theorems~\ref{thm:direct}, \ref{thm:stab}, \ref{thm:equiv}, \ref{thm:A} and
\ref{thm:B} are stated uniformly in $n$; the dimension enters only through the constants, in particular $c_3,c_4$. Pointwise arguments
such as \cite[(12)--(15)]{DD21} for $n=1$ cannot be carried over to $n\ge4$.

\item \emph{Role of the hypotheses.} The absence of a zero-order term ($c\equiv0$) gives $Y\ge0$ with  exact normalization
\eqref{eq:norm}; hence the explicit constants $\lambda_0 E_{\alpha}(\|q\|_0T^\alpha)$ in \eqref{eq:supbound} and in Theorem~\ref{thm:C}, and the
subtraction trick in Lemma~\ref{lem:ops}(a),(c). Hypothesis (A2) is used only in the reduction of Section 4, through
the invariance of the problem under translations in $x_n$ (Lemma~\ref{lem:transl}). Condition \eqref{eq:g0} is needed to solve
\eqref{eq:traceII} for $q$. Hypothesis (A3) is not needed for stability and uniqueness; without it the tangential derivatives
$\partial_{x_i}w$ enter the trace condition, and the singularity of $\nabla_xY$ lowers the order of the Volterra operator from $\alpha$
to $\alpha/2$, which is still integrable.

\item \emph{In what sense the stability is conditional.} Estimate \eqref{eq:stabA} holds on $Q_R$ with a constant that grows with $R$: by
\eqref{eq:Cstar} and $E_{\alpha/2}(s)\sim\frac2\alpha e^{s^{2/\alpha}}$ as $s\to\infty$,
$C_*$ grows like $\exp(cB^{2/\alpha}T)$, where $B$ increases with $R$ through $W$, $V$ and $R$
itself; thus the estimate is uniform only on the a priori set $Q_R$. It is a Lipschitz estimate, but with respect to the data distance $\varrho$ in \eqref{eq:dist}, which contains
${}^CD_t(g-\tilde g)$ and $L'(g-\tilde g)$. In terms of $\|g-\tilde g\|_0$ alone no such estimate can be expected: recovering $q$ requires
evaluating the unbounded operator ${}^CD_t-L'$ on the measurement, which is the source of ill-posedness. With noisy data $g^\delta$,
$\|g^\delta-g\|_0\le\delta$, one must regularize $\mathcal G$ (e.g.\ by mollification in $x'$ and $t$); combined with an a priori bound on a
stronger norm of $g$, interpolation would turn \eqref{eq:stabA} into a H\"{o}lder-type estimate in weaker norms. We do not pursue this here.

\item \emph{Global existence.} Theorems~\ref{thm:Dgen} and \ref{thm:D} give a solution only
for  a short interval $[0,T^*]$. Solutions on different intervals are consistent: if $q'$ and
$q''$ solve the inverse problem on $[0,T']$ and $[0,T'']$, $T'\le T''\le T$, then $q'=q''$ on
$\mathbb{R}^{n-1}\times[0,T']$. Indeed, the restriction to $\mathbb{R}^n\times[0,T']$ of the regular solution
of \eqref{eq:main}, \eqref{eq:init} with coefficient $q''$ belongs to $\mathcal{U}_{T'}(\mathbb{R}^n)$ and
satisfies the equation, hence, by Theorem~\ref{thm:direct}(iii), it is the regular solution
on $[0,T']$; therefore $q''|_{[0,T']}$ solves the inverse problem on $[0,T']$, and
Theorem~\ref{thm:B} with $T$ replaced by $T'$ gives $q'=q''$ there.
Extending a local solution to the whole interval $[0,T]$ would require, in addition, a local existence result on intervals $[T',T'+h]$ that takes into account the memory of the fractional derivative, and an a priori bound $\|q\|_0\le R$ on $[0,T]$; for large data both questions are open. Under (A3) and smallness of $\|a_{nn}/g\|_\ell$ relative to the data, conditions \eqref{eq:Tstar} can be satisfied with $T^*=T$.
\end{enumerate}

\end{document}